\documentclass[14pt]{amsart}
\usepackage[english]{babel}
\usepackage{graphicx}
\usepackage[T1]{fontenc}
\usepackage{color}
\usepackage{bigints}
\usepackage{soul}
\usepackage{mathtools}
\providecommand{\gtrless}{
	\mathrel{
		\smash{
			\vcenter{
				\offinterlineskip 
				\ialign{
					\hfil##\hfil\cr 
					$>$\cr 
					\noalign{\kern-.3ex}
					$<$\cr 
				}
			}
		}
		\vphantom{>}
	}
}

\newtheorem{Thm}{Theorem}
\newtheorem{Prop}[Thm]{Proposition}

\newtheorem{Lem}{Lemma}

\newtheorem*{Prop*}{Proposition}
\newtheorem*{Cor*}{Corollary}
\newtheorem*{Thm*}{Theorem}
\usepackage[backref=page]{hyperref}
\usepackage[utf8]{inputenc}
\usepackage[english]{babel}
\usepackage{etoolbox}
\DeclarePairedDelimiter\floor{\lfloor}{\rfloor}
\apptocmd{\sloppy}{\hbadness 10000\relax}{}{}
\hypersetup{
	colorlinks=true,
	linkcolor=blue,
	urlcolor=cyan,
}

\begin{document}
	\title{Boundary Control and Calder\'on type Inverse Problems in Non-local heat equation}
	\author[Saumyajit Das]{Saumyajit Das}
	\address{Department of Mathematics, Harish-Chandra Research Institute, A CI of Homi Bhabha National Institute, Chhatnag Road, Jhunsi, Allahabad 211 019, India.}
	\email{saumyajit.math.das@gmail.com}
	\begin{abstract}
		We examine various density results related to the solutions of the non-local  heat equation at a specific time slice, focusing on two distinct models: one with homogeneous Dirichlet boundary condition and the other with singular boundary data. In both the cases, we assume the non-local exponent $a\in(\frac{1}{2},1)$. We explore both the qualitative and quantitative aspects of the approximations. However, to obtain the approximation result for singular boundary data, we assume the potential to be small, non-negative, and exhibit only mild growth. The smallness of the potential is explicit in our study and depends on the domain and the dimension only. Additionally, we address Calderón-type inverse problems for these parabolic models, where we recover the potentials by analyzing the solutions either on the boundary or at a particular time slice. In both the density results and the Calderón type inverse problems, the Pohozaev identity plays a crucial role. Finally, in the last section, we apply the Pohozaev identity to a specific elliptic eigenvalue problem and demonstrate that the eigenfunctions, when divided by an appropriate power of the distance function, can not vanish on any non-empty open subset of the boundary. This particular eigenvalue problem does not need any restriction on the non-local exponent.
	\end{abstract}
	\maketitle
	\bibliographystyle{alpha}

	\section{Introduction}
	In this article, we address various qualitative and quantitative approximation properties of solutions to certain  non-local heat equation. Specifically, we consider two equations where one has homogeneous Dirichlet boundary condition and the other has singular boundary data. For the singular boundary data, we analyze the approximation properties through the Pohozaev identity [see \cite{ros2014pohozaev}]. Using the same Pohozaev identity and density results, we also recover the potentials for both the systems by analyzing the solutions either at a specific time slice or on the boundary.  We consider the following two equations:
	\begin{equation}\label{varying initial data}
		\left\{
		\begin{aligned}
			\partial_t u_{f}+(-\Delta)^a u_f+ qu_f=&0 \qquad \qquad  \mbox{in}\ (0,T)\times\Omega\\
			u_f=&0 \qquad \qquad  \mbox{in} \ (0,T)\times\mathbb{R}^N\setminus\Omega\\
			u_f(0,\cdot)=&f \qquad \qquad  \mbox{in} \ \Omega,
		\end{aligned}
		\right .
	\end{equation}
	and
	\begin{equation}\label{varying boundary data}
		\left\{
		\begin{aligned}
			\partial_t u_{F}+(-\Delta)^a u_F+ qu_F=&0 \qquad \qquad  \mbox{in}\ (0,T)\times\Omega\\
			u_F=&0 \qquad \qquad  \mbox{in} \ (0,T)\times\mathbb{R}^N\setminus\overline{\Omega}\\
			\lim\limits_{\substack{x\to y\\ y\in\partial\Omega}}\frac{u_F}{d^{a-1}}(y)=& F \qquad \qquad \mbox{on} \ (0,T)\times\partial\Omega\\
			u_F(0,\cdot)=&0 \qquad \qquad  \mbox{in} \ \Omega.
		\end{aligned}
		\right .
	\end{equation}
	Here $\Omega\subset\mathbb{R}^N$ is a bounded $C^{1,1}$ domain with boundary $\partial\Omega$ and $a\in(\frac{1}{2},1)$. The functions $u_f, u_F\in\mathrm{L}^2\left((0,T)\times\Omega\right)$ are the unknowns and the potential $q\in C_c^{\infty}\left(\Omega\right)$. Furthermore, for \eqref{varying initial data}, we assume the initial condition $f\in C_c^{\infty}(\Omega)$  and for \eqref{varying boundary data} we assume the singular boundary data $F\in C_c^{\infty}((0,T)\times\partial\Omega)$. The distance function $\displaystyle{d:\Omega\to \mathbb{R}_{\geq 0}}$ is the following:
	\[
	d(x):=\inf\limits_{y\in\partial\Omega}\Vert x-y\Vert_2,
	\]
	where $\Vert \cdot\Vert_2$ is the classical Euclidean norm in $\mathbb{R}^N$. To ensure uniqueness, we choose $q$ such that zero is not an eigenvalue of either equation \eqref{varying initial data} or \eqref{varying boundary data}, corresponding to the cases where $f\equiv 0$ or $F\equiv 0$, respectively. Regarding approximation, in \cite{Ruland2017QuantitativeAP}, the authors investigate a qualitative and quantitative approximation property of solutions to a specific non-local heat equation, where they vary the $N+1$ Dirichlet data in the exterior domain. They showed that the solution space is dense in $\mathrm{L}^2$ in the interior by leveraging the unique continuation property: if a function $u$ satisfies $u = (-\Delta)^a u = 0$ in some open set $W \subset \mathbb{R}^N$, then it must vanish identically, i.e., $u \equiv 0$ \cite{ghosh2020calderon}. The precise result is as follows:
	\begin{Thm}[\cite{Ruland2017QuantitativeAP}]
		Let $a\in(0,1)$ and $B_1$ is the unit ball in $\mathbb{R}^N$ centered at origin. Let $W\subset \mathbb{R}^N$ be a non-empty open set such that $\overline{B}_1\cap \overline{W}$ is empty. Let for $f\in C_c^{\infty}((-1,1)\times W)$, the function $u_f:(-1,1)\times\mathbb{R}^N\to \mathbb{R}$ be the solution of the following equation:
		\begin{equation}\label{exterior}
			\left\{
			\begin{aligned}
				\partial_t u_f+(-\Delta)^a u_f=&0 \qquad \qquad  \mbox{in}\ (-1,1)\times B_1\\
				u_f=&f \qquad \quad \ \    \  \mbox{in} \ (-1,1)\times\mathbb{R}^N\setminus B_1\\
				u_f(-1,\cdot)=&f \qquad \qquad  \mbox{in} \ B_1.
			\end{aligned}
			\right .
		\end{equation}
		Then the solution set 
		\[
		\{u_f: \ u_f\  \mbox{satisfies}\ \eqref{exterior}, \, f\in C_c^{\infty}((-1,1)\times W)\}
		\]
		is dense in $\mathrm{L}^2((-1,1)\times B_1)$.
	\end{Thm}
	The quantitative analysis enables the authors to derive an estimate for the cost of approximation, as described in \cite{Ruland2017QuantitativeAP}. This requires the use of variational techniques, as presented in \cite{Ruland2017QuantitativeAP}, \cite{zuazua}, and \cite{lions1992remarks}. A qualitative version of similar approximation properties for the non-local elliptic model is presented in \cite{ghosh2020calderon}, where the authors investigate the Calder{\'o}n type problem for the fractional heat equation. A comparable result in the $C^k$ space is established in \cite{dipierro2019local}, where the authors demonstrate that any $C^k$ function can be approximated by the smooth solutions to the non-local heat equation, with the approximation property governed by the fractional component of the equation only. However, in this article, we vary either the $N$ dimensional initial data or boundary data, and obtain $\mathrm{L}^2$ approximation on an $N$-dimensional interior slice. In particular, we vary the initial data $f\in C_c^{\infty}(\Omega)$ for the equation with homogeneous Dirichlet boundary data \eqref{varying initial data}, and the singular boundary data $F\in C_c^{\infty}((0,T)\times\partial\Omega)$ for the equation \eqref{varying boundary data}. We demonstrate that for a fix time slice $T>0$, the solution sets $\displaystyle{\{u_f: f\in C_c^{\infty}(\Omega)\}}$ and $\displaystyle{\{u_F: F\in C_c^{\infty}((0,T)\times\partial\Omega)\}}$ corresponding to the equations \eqref{varying initial data} and \eqref{varying boundary data} respectively, are dense in $\mathrm{L}^2(\Omega)$. The application of the integration by parts formula (in the $\mathrm{L}^2$ sense) imposes the following constraint $a\in(\frac12,1)$ on the non-local exponent.  We think such results and techniques executed here could be extended to the classical heat equation i.e., $a=1$.
	\newline
	These results are partially motivated by the stability results for Calderón type inverse problems for the non-local heat equation, as studied in \cite{Ruland2017QuantitativeAP}. Our work can be seen as a continuation of the contributions made in \cite{Ruland2017QuantitativeAP} and \cite{zuazua}, extending the analysis to the boundary controllability in the presence of singular boundary data. Additionally, the study in \cite{ros2014pohozaev} serves as a key motivation for exploring boundary controllability and Calderón type inverse problems in the context of both the singular boundary and the initial data problem. 
	\newline
	The density argument for the case of the singular boundary data \eqref{varying boundary data} requires the Pohozaev estimate, as demonstrated in \cite{ros2014pohozaev}. For small potential, we can prove the density result using observability estimate as described in \cite{biccari2025boundary}. Using the density results and the same arguments presented in the proof of observability estimate, we solve the Calder\'on type inverse problems for both the  non-local heat  equations \eqref{varying initial data}, and \eqref{varying boundary data}, where we analyze the solutions at a particular time slice $T>0$ for \eqref{varying initial data}, and for \eqref{varying boundary data} we analyze the solutions on the boundary. The study of the Calder\'on type inverse problems for the  non-local elliptic equations can be found in \cite{ghosh2017calderon}, \cite{ghosh2020uniqueness}, \cite{salo2017fractional}, \cite{ghosh2021non}. The main results in this article are the following:
	\begin{Thm}\label{Density result initial data}
		Let $a\in(\frac{1}{2},1)$ and let $f\in C_c^{\infty}(\Omega)$. Let $u_f\in\mathrm{L}^2\left((0,T)\times\Omega\right)$ be the solution of \eqref{varying initial data}. Then, for a fix time $\Tilde{T}>0$, the set $\{u_f(\Tilde{T},\cdot): f\in C_c^{\infty}(\Omega)\}$ is dense in $\mathrm{L}^2(\Omega)$.
	\end{Thm}
	\begin{Thm}\label{Quantitative Density result boundary data}
		Let $a\in(\frac{1}{2},1)$ and let $F\in C_c^{\infty}((0,T)\times\partial\Omega)$. Let $\displaystyle{\theta< \frac{1}{2}\left(1+C_{HS}\left(\frac N2+R\right)\right)^{-1}}$, where $C_{HS}$ is the Hardy-Sobolev constant on the domain $\Omega$ and $R:=\max\{\Vert x\Vert, x\in\Omega\}$. Let the potential $q$ is non-negative and $|q(x)|,|\nabla q(x)|\leq \theta$ for all $x\in\Omega$. Let $u_F\in\mathrm{L}^2\left((0,T)\times\Omega\right)$ be the solution of \eqref{varying boundary data}. Then, for a fix time $\Tilde{T}>0$, the set 
		\[
		\left\{u_F(\Tilde{T},\cdot): F\in C_c^{\infty}((0,T)\times\partial\Omega)\right\}\]
		is dense in $\mathrm{L}^2(\Omega)$.
	\end{Thm}	
	\begin{Thm}\label{inverse result initial data}
		Let $a\in(\frac{1}{2},1)$ and let $u_1,u_2\in \mathrm{L}^2((0,T);H^a(\Omega))$ satisfy: for $i=1,2$
		\begin{equation}\label{varying initial data inverse result}
			\left\{
			\begin{aligned}
				\partial_t u_{i}+(-\Delta)^a u_i+ q_iu_i=&0 \qquad \qquad  \mbox{in}\ (0,T)\times\Omega\\
				u_i=&0 \qquad \qquad  \mbox{in} \ (0,T)\times\mathbb{R}^N\setminus\Omega\\
				u_i(0,\cdot)=&f \qquad \qquad  \mbox{in} \ \Omega,
			\end{aligned}
			\right .
		\end{equation}
		where $q_1,q_2\in C_c^{\infty}(\Omega)$ and $f\in C^{\infty}_c(\Omega)$. Furthermore,  $q_1,q_2$ be such that zero isn't an eigenvalue to the same equation with zero initial data. If 
		\[
		\frac{u_1}{d^a}(t,x)= \frac{u_2}{d^a}(t,x), \quad \forall \, (t,x)\in(0,T]\times\partial\Omega \ \mbox{and}\ \forall \, f\in C_c^{\infty}(\Omega),
		\]
		then
		\[
		q_1\equiv q_2.
		\]
	\end{Thm}
	\begin{Thm}\label{inverse result boundary data}
		Let $a\in(\frac12,1)$ and let $u_1,u_2\in \mathrm{L}^2((0,T);H^{a-1(2a)}(\overline{\Omega}))$ satisfy: for $i=1,2$	
		\begin{equation}\label{varying boundary data inverse result}
			\left\{
			\begin{aligned}
				\partial_t u_{i}+(-\Delta)^a u_i+ q_iu_i=&0 \qquad \qquad \ \mbox{in}\ (0,T)\times\Omega\\
				u_i=&0 \qquad \qquad \  \mbox{in} \ (0,T)\times\mathbb{R}^N\setminus\overline{\Omega}\\
				\lim\limits_{\substack{x\to y\\ y\in\partial\Omega}}\frac{u_i}{d^{a-1}}(y)=& F \qquad \qquad \mbox{on} \ (0,T)\times\partial\Omega\\
				u_i(0,\cdot)=&0 \qquad \qquad  \ \mbox{in} \ \Omega,
			\end{aligned}
			\right .
		\end{equation}
		where $q_1,q_2\in C_c^{\infty}(\Omega)$ and $F\in C_c^{\infty}((0,T)\times\partial\Omega)$. Furthermore,  $q_1,q_2$ be such that zero is not an eigenvalue to \eqref{varying initial data inverse result} with zero boundary data. If
		\[
		u_1(T,\cdot)\equiv u_2(T,\cdot) \qquad \forall \, F\in C_c^{\infty}((0,T)\times\partial\Omega)
		\]
		then
		\[
		q_1\equiv q_2.
		\]
	\end{Thm}
	In fact, we will show that the statements of the theorem \ref{inverse result initial data} and the theorem \ref{inverse result boundary data} are equivalent. Next we describe the solution spaces, well posedness of the solutions and various regularity results. Many of the cases, the notion of the solutions for the non-local heat equation can be extended to the non-local  heat equation with the perturbation potential $q$, as the potential $q$ lies in $C_c^{\infty}(\Omega)$ \cite{pazy2012semigroups}, \cite{bogdan2012estimates}. 
	
	\subsection{Preliminaries: Fractional Laplacian and fractional Sobolev space}	
	Let $0< a< 1$. The fractional Laplacian operator $(-\Delta)^a$ is defined over the space of Schwartz class functions $\mathcal{S}(\mathbb{R}^N)$ as
	\[
	(-\Delta)^a u(x):= \mathcal{F}^{-1}\left\{|\xi|^{2a}\Hat{u}(\xi)\right\}, \quad \forall \, x\in\mathbb{R}^N,
	\]
	where $\Hat{.}$ and $\mathcal{F}^{-1}$ denotes the Fourier and the inverse Fourier transformations respectively. There are many equivalent definitions of fractional Laplacian [see \cite{kwasnicki2017ten}]. One of such equivalent definition we use throughout this article is  given by the Cauchy principle value: for $0\leq a\leq 1$
	\[
	(-\Delta)^a u(x)= C_{N,a} \ \mbox{p.v} \ \int_{\mathbb{R}^N} \frac{u(x)-u(y)}{| x-y|^{N+2a}} \, \rm{d}y, \qquad \forall \, x\in\mathbb{R}^N, 
	\] 
	where $\displaystyle{C_{N,a}=\frac{4^a\Gamma\left(\frac{N}{2}+a\right)}{\pi^{\frac{N}{2}}\Gamma(-a)}}$. Throughout the article, $\Gamma$ stands for the usual Gamma function. We define the following Sobolev space: for $s\in(0,1)$
	\[
	H^s(\mathbb{R}^N):= \{ u\in\mathbb{S}'(\mathbb{R}^N): \langle \xi\rangle^{s}\Hat{u}\in \mathrm{L}^2(\mathbb{R}^N)\},
	\]
	where $\displaystyle{\langle \xi\rangle=(1+\Vert \xi\Vert_{2}^2)^{\frac{1}{2}}}$ and $\mathbb{S}'(\mathbb{R}^N)$ is the space of tempered distributions in $\mathbb{R}^N$, equipped with the norm
	\[
	\Vert u\Vert_{H^s(\mathbb{R}^N)}= \Vert \langle \xi\rangle^s \Hat{u}\Vert_{\mathrm{L}^2(\mathbb{R}^N)}.
	\]
	The fractional Laplacian extends as a bounded linear map 
	\[
	(-\Delta)^a: H^s(\mathbb{R}^N) \to H^{s-2a}(\mathbb{R}^N).
	\]
	We also introduce few spaces that are essential for describing the solution spaces and various regularity results throughout the article.
	\begin{align*}
		H^s(\Omega):=& \{u= v|_{\Omega}: v \in H^s(\mathbb{R}^N)\}\\
		C^s(\Omega):=& \{u\in C(\Omega): |u(x)-u(y)|\leq M|x-y|^s, \, \forall \, x,y\in \Omega, M\geq0\}\\
		C^s_c(\mathbb{R}^N):=& \{ u\in C^s(\mathbb{R}^N): \, \mbox{supp}\, u \ \mbox{compactly contained in}\ \mathbb{R}^N\},
	\end{align*}
	where $\Omega$ is a bounded $C^{1,1}$ domain and $s\in(0,1)$. We denote $v|_{\Omega}$ as the restriction of the function $v$ in $\Omega$. $H^s(\Omega)$ is a Sobolev space equipped with the norm 
	\[
	\Vert u\Vert_{H^s(\Omega)}= \Vert v\Vert_{H^s(\mathbb{R}^N)}= \Vert \langle \xi\rangle^s \Hat{v}\Vert_{\mathrm{L}^2(\mathbb{R}^N)}.
	\] 
	The definition of the space $C^s(\Omega)$ can be extended to the set when $\Omega$ is closed i.e., $\Omega=\overline{\Omega}$ or $\Omega=\mathbb{R}^N$. We define the following seminorm associated with the space $C^s(\Omega)$:
	\[
	[u]_{C^s(\Omega)}:= \inf_{M\geq 0} \left \{M: |u(x)-u(y)|\leq M |x-y|^s, \, \forall \, x,y\in\Omega \right\}.
	\]
	The spaces $C_c^s(\mathbb{R}^N)$ and $C^s(\overline{\Omega})$ $(\Omega$ is bounded$)$  can be endowed with the following norms respectively
	\begin{align*}
		\Vert u\Vert_{C^s_c(\mathbb{R}^N)}:= &\sup\limits_{x\in \mathbb{R}^N} |u(x)|+ [u]_{C^s(\mathbb{R}^N)},\\
		\Vert u\Vert_{C^s(\overline{\Omega})}:= &\sup\limits_{x\in\overline{\Omega}} |u(x)|+ [u]_{C^s(\Omega)}.
	\end{align*}
	Note that $C^s(\overline{\Omega})$ is a Banach space with respect to the norm described above.
	\newline
	The definition of $C^s(\Omega)$ $($or $C_c^s(\mathbb{R}^N))$ can be extended to any $\beta>0$. We can express $\beta=k+s$, where $k\in\mathbb{N}\cup\{0\}$ and $s\in(0,1)$. The space $C^{\beta}(\Omega)$ is the following
	\[
	C^{\beta}(\Omega):=\left\{ D^{\alpha}u\in C(\Omega): |D^{k}u(x)-D^ku(y)|\leq M|x-y|^s, \, \forall \, x,y\in\Omega, M\geq0\right\}
	\]
	where $\displaystyle{\alpha=(\alpha_1,\cdots,\alpha_N)\in(\mathbb{N}\cup\{0\})^N}$ and $\displaystyle{\sum\limits_{i=1}^{N}\alpha_i \leq k}$ and $D^{\alpha}$ is the classical derivative of order $\alpha$.
	
	\subsection{Wellposedness of the solution}
	Let us define the following space: for $s\in(0,1)$
	\[
	\mathrm{L}^1\left(\Omega,d^s\right):=\left\{f: \int_{\Omega} fd^s dx <+\infty\right\}. 
	\]
	The wellposedness of the two equations has been studied in \cite{chan2022singular}. It has been shown that under the assumption $f\in \mathrm{L}^1(\Omega,d^a)$ there exists a solution $u_f\in \mathrm{L}^1((0,T); \mathrm{L}^1(\Omega, d^a))$ to \eqref{varying initial data}. In the same article \cite{chan2022singular}, it has been shown that under the assumption $F\in\mathrm{L}^1((0,T)\times\partial\Omega)$ there exists a solution $u_f\in \mathrm{L}^1((0,T); \mathrm{L}^1(\Omega, d^a))$. We choose $q$ such that zero is not an eigenvalue for both the equations \eqref{varying initial data}, \eqref{varying boundary data} i.e., if $u_0\in\mathrm{L}^1((0,T); \mathrm{L}^1(\Omega, d^a))$ satisfies either \eqref{varying initial data} or \eqref{varying boundary data}, corresponding to $f\equiv 0$ or $F\equiv 0$ respectively, then $u_0\equiv 0$. This ensures the uniqueness of the solutions also. 
	\newline
	Let $P(t,x,y)$ be the fractional heat semigroup corresponding to the operator $\partial_t+(-\Delta)^a+q$ with homogeneous Dirichlet boundary condition and $\delta_x$(Dirac delta) as initial condition. Then the quantity $\displaystyle{\lim\limits_{\substack{y\to w\\ w\in\partial\Omega}}\frac{P(t,x,y)}{d^{a-1}(y)}}$ is a well defined quatity [see \cite{chan2022singular}]. Well defineness of that particular quantity comes from the existence of the same limit function for the corresponding elliptic Green function with homogeneous Dirichlet boundary condition [see \cite{abatangelo2023singular}].  The solution $u_f$ corresponding to \eqref{varying initial data} or $u_{F}$ corresponding to \eqref{varying boundary data} can be uniquely expressed as 
	\begin{align*}
		u_f(t,x)= \int_{\Omega}P(t,x,y)f(y) dy\\
		u_{F}(t,x)= \int_{0}^{t}\int_{\partial\Omega} \lim\limits_{\substack{y\to w\\ w\in\partial\Omega}}\frac{S(t-s,x,y)}{d^{a-1}(y)} F(s,w) \, \rm{d}w \, \rm{d}s. 
	\end{align*}
	We can further show if $f\in\mathrm{L}^2(\Omega)$, then the solution to \eqref{varying initial data} belongs to the following space 
	\[
	u_f\in\mathrm{L}^2((0,T);  H^a(\mathbb{R}^N)),
	\] 
	given by the same expression above [see \cite{chan2022singular},  \cite{grubb2017fractional}]. The smoothness of the fractional heat semigroup for $t>0$ and away from boundary $\partial\Omega$ \cite{fernandez2016boundary}, \cite{dong2023time} implies $u_f \in C^{\infty}\left((0,T)\times\Omega\right)$.   The fractional heat semigroup can be expressed in the following way:
	\[
	P(t,x,y)=\sum\limits_{n\in\mathbb{N}}e^{-t\lambda_n}\phi_n(x)\phi_n(y)
	\]
	where $\lambda_1<\lambda_2\leq \dots\leq \lambda_n\leq\dots+\infty$ are the eigenvalues and $\phi_n$'s are the normalized eigenfunctions of the following homogeneous Dirichlet  problem
	\begin{equation}\label{eigen function equation}
		\left \{
		\begin{array}{ll}
			((-\Delta)^a+q)\phi_n=\lambda_n\phi_n \qquad \qquad&\mbox{in}\ \Omega  \\
			\phi_n=0 & \mbox{in}\ \mathbb{R}^N\setminus \Omega. 
		\end{array}
		\right .
	\end{equation}
	Various interior regularity results corresponding to \eqref{varying initial data} can be found in \cite{ros2018higher}, \cite{fernandez2016boundary}. In \cite{fernandez2016boundary}, authors have proved the following results corresponding to \eqref{varying initial data}:
	\begin{itemize}
		\item [$($a$)$] for each $t_0>0$
		\begin{align}\label{H0}
			\sup\limits_{t\geq t_0}\Vert u_f(t,\cdot)\Vert_{C_c^a(\mathbb{R}^N)}\leq C_1(t_0)\Vert f\Vert_{\mathrm{L}^2(\Omega)},
		\end{align}
		\item [$($b$)$] for each $t_0>0$
		\begin{align}\label{H1}
			\sup\limits_{t\geq t_0}\left\Vert \frac{u_f(t,\cdot)}{d^a}\right\Vert_{C^{a-\epsilon}\left(\overline{\Omega}\right)}\leq C_1(t_0)\Vert f\Vert_{\mathrm{L}^2(\Omega)},
		\end{align}
		\item [$($c$)$] for each $t_0>0$
		\begin{align}\label{H4}
			\sup\limits_{t\geq t_0} \left\Vert \frac{d^j}{dt^j}u_f(t,\cdot)\right\Vert_{C^a_c(\mathbb{R}^N)} \leq C_{1,j}(t_0) \Vert f\Vert_{\mathrm{L}^2(\Omega)}
		\end{align}
	\end{itemize}
	for any $\epsilon\in(0,a)$. Here $C_1(t_0)$,$C_{1,j}(t_0)$ are positive constants. Authors used various eigenfunction estimates  corresponding to \eqref{eigen function equation} [see \cite{grubb2015spectral}, \cite{grubb2015fractional}, \cite{abatangelo2023singular}]. The following interior regularity results can be found in  \cite{fernandez2016boundary}.
	\begin{Thm}[\cite{fernandez2016boundary}]
		Let $u_f\in\mathrm{L}^2\left((0,T); H^{a}(\Omega)\right)$ be the solution to \eqref{varying initial data} with $f\in \mathrm{L}^2(\Omega)$. Then the following holds.
		\begin{itemize}
			\item [$\bullet$:] For $t>0$, $u_f(t,\cdot)\in C^a(\mathbb{R}^N)$ and, for every $\beta\in[a,1+2a)$, $u_f(t,\cdot)$ is of class $C^{\beta}(\Omega)$ and 
			\begin{align}\label{H2}
				[u_f(t,\cdot)]_{C^{\beta}(\Omega_{\delta})} \leq C \delta^{a-\beta}, \forall \delta\in(0,1).
			\end{align}
			\item [$\bullet$:] The function $\displaystyle{\frac{u_f(t,\cdot)}{d^a}}\Big|_{\Omega}$ can be continuously extended to $\overline{\Omega}$. Moreover, there exists $\alpha\in(0,1)$ such that $\displaystyle{\frac{u_f}{d^a}\in C^{\alpha}(\overline{\Omega})}$. In addition, for $\beta\in[\alpha,a+\alpha]$, the following estimate holds
			\begin{align}\label{H3}
				\left[\frac{u_f}{d^a}\right]_{C^{\beta}(\Omega_{\delta})} \leq C \delta^{a-\beta}, \forall \delta\in(0,1).
			\end{align}
		\end{itemize}
		Here $\Omega_{\delta}:=\{x\in\Omega: d(x,\partial\Omega)\geq\delta\}$ and $[\cdot]_{C^{\beta}(\Omega_{\delta})}$ denotes the H\"older seminorm:
		\[
		\left[\frac{u_f(t,x)}{d^a}\right]_{C^{\beta}(\Omega_{\delta})} =\sup\limits_{x,y\in C^{\beta}(\Omega_{\delta}), x\neq y} \frac{\vert D^{\floor{\beta}}u_f(t,x)-D^{\floor{\beta}}u_f(t,y)\vert}{\vert x-y\vert^{\beta-\floor{\beta}}}
		\]
		where $\floor{\beta}$, denotes the greatest integer function i.e $\floor{\beta}:=\sup\{z\in\mathbb{Z}: z\leq \beta\}$. 
	\end{Thm}
	With \eqref{H0}, \eqref{H1}, \eqref{H2}, \eqref{H3} in hand, the solution $u_f$ corresponds to \eqref{varying initial data} satisfies the fractional version of Pohozaev identity. For all $t>0$:
	\begin{align}\label{Phozaev identity}
		\int_{\Omega}(x\cdot \nabla u_f)(-\Delta)^a u_f \, \rm{d}x= \frac{2a-N}{2}\int_{\Omega}&u_f(-\Delta)^a u_f \, \rm{d}x\\
		&- \frac{\Gamma(1+a)^2}{2}\int_{\partial\Omega}\left(\frac{u_f}{d^a}\right)^2(x\cdot\nu) \, \rm{d}\sigma,\nonumber
	\end{align}
	where $\nu$ is the unit outward normal to $\partial\Omega$ at the point $x$ and $\Gamma$ is the gamma function. Proof of this statement is established in the articles \cite{ros2014pohozaev}, \cite{fernandez2016boundary}. 
	\newline
	Next we introduce $\mu$ transmission spaces.
	\subsection{$\mu$ transmission spaces}
	All the notations and definitions here are well explained in \cite{grubb2014local}, \cite{grubb2015fractional}, \cite{grubb2016regularity}, \cite{grubb2016integration}, \cite{grubb2018green}, \cite{Grubb2020ExactGF}. We only recall some important definitions and concepts here. 
	\newline
	We consider operators acting on the upper or lower half subspaces of $\mathbb{R}^N$ denoted by $\displaystyle{\mathbb{R}^N_{\pm}:=\{(x_1,\cdots,x_N)\in\mathbb{R}^N: x_N\gtrless 0\}}$. Similarly with respect to the boundary $\partial\Omega$ we divide a $C^{1,1}$ domain $\Omega$ into two parts. One part is the domain $\Omega$ itself and the other is the complement of it's closure $\displaystyle{\mathbb{R}^N\setminus\overline{\Omega}}$. We denote $r^{\pm}$ as the restriction operator from $\mathbb{R}^{N}$ to  $\displaystyle{\mathbb{R}^N_{\pm}}$ $\Big($or from $\mathbb{R}^{N}$ to $\Omega$ respectively $\displaystyle{\mathbb{R}^N\setminus\overline{\Omega}}\Big)$. The extension by zero from $\displaystyle{\mathbb{R}^N\setminus\overline{\Omega}}$ to $\mathbb{R}^N$ $\Big($or from $\Omega$ respectively $\displaystyle{\mathbb{R}^N\setminus\overline{\Omega}}$ to $\mathbb{R}^{N}\Big)$ is denoted by $e^{\pm}$.
	\newline
	Let us denote a function $\tilde{d}$ which is of the form $d(x):=\text{dist}(x,\partial\Omega)$ near boundary and extended smoothly into a positive function in $\Omega$. We define the space
	\begin{align}\label{density in transmission space}
	\mathcal{E}_{\mu}\left(\overline{\Omega}\right):= e^{+}\left\{u(x)=\tilde{d}^{\mu}v(x): v\in C^{\infty}\left(\overline{\Omega}\right)\right\},
	\end{align}
	for Re$(\mu)>-1$. For general $\mu$, we refer readers to \cite{grubb2015fractional}, \cite{seeley1971norms}. Next we define pseudodifferential operator ($\psi$do) $\mathcal{P}$ on $\mathbb{R}^N$, which is defined from a symbol $p(x,\xi)$ on $\displaystyle{\mathbb{R}^N\times \mathbb{R}^N}$ in the following way:
	\begin{align*}
		\mathcal{P}u=p(x,D)u=\text{Op}((p(x,\xi))u= (2\pi)^{-N}\int e^{ix\dot\xi}p(x,\xi)\hat{u} \, \rm{d}\xi=\mathcal{F}_{\xi\to x}^{-1}((p(x,\xi)\hat{u}(\xi)),
	\end{align*}
	where $\hat{u}$ is the Fourier transformation of $u$. For the theory of calculus, we refer to books and notes like \cite{hormander1963linear}, \cite{hormander1966seminar}, \cite{taylor2006pseudo}, \cite{grubb2008distributions}. 
	We define the symbol space $S^m_{1,0}\left(\mathbb{R}^N\times \mathbb{R}^N\right)$ as follows:
	\begin{equation*}
		S^m_{1,0}:=
		\left\{
		\begin{aligned}
			p(x,\xi)\in C^{\infty}\left(\mathbb{R}^N\times \mathbb{R}^N\right): &\ \vert \partial_{x}^{\beta}\partial_{x}^{\alpha} p(x,\xi)\vert \sim O\left(\langle\xi\rangle^{m-|\alpha|}\right),\\
			&  \forall \alpha,\beta\in\mathbb{N}\cup\{0\}, \ \mbox{for some}\ m\in\mathbb{C}
		\end{aligned}
		\right \}
	\end{equation*}
	where $\displaystyle{\langle\xi\rangle:=\left(1+\Vert \xi\Vert^2_2\right)^{\frac{1}{2}}}$ and $O$ is the big-O function captures the order. The exponent $m$ is called the order of $\mathcal{P}($and $p)$. Note that the operator $\left((-\Delta)^{a}+q\right)$ is a $\psi$do operator of order 2a. Depending on the context, the operator $\mathcal{P}_+=r^{+}\mathcal{P}e^{+}$ denotes the truncation of $\mathcal{P}$ to $\mathbb{R}^{N}_{+}$ or to $\Omega$. We define the following $\mathrm{L}^2$ spaces:
	\begin{align*}
		&H^s\left(\mathbb{R}^N\right):= \left\{ u\in\mathcal{S}'\left(\mathbb{R}^N\right)|\mathcal{F}^{-1}\left(\langle\xi\rangle^s\hat{u}\right)\in \mathrm{L}^2\left(\mathbb{R}^N\right)\right\},
		\\
		&\dot{H}^s\left(\overline{\Omega}\right):= \left\{ u\in H^s\left(\mathbb{R}^N\right)| \ \mbox{supp}\ u\subset \overline{\Omega}\right\},
		\\
		&\overline{H}^s(\Omega):= \left\{ u\in \mathcal{D}'(\Omega)| u=r^{+}U \ \mbox{for a}\  U\in H^s\left(\mathbb{R}^N\right)\right\}.
	\end{align*}
	Next we define order reducing operator on $\mathbb{R}^N$: for $t\in\mathbb{R}$
	\[
	\Xi^{t}_{\pm}:= Op\left(\chi^t_{\pm}\right), \quad \chi^t_{\pm}:=\left( \langle \xi'\rangle\pm i\xi_N\right)^t \ \mbox{where} \ \xi=(\xi',\xi_N), \xi':=(\xi_1,\dots,\xi_{N-1}).
	\]
	These symbols extend analytically in $\xi_N$ to Im $\displaystyle{\xi_N\gtrless 0}$. Hence by Paley-Wiener theorem $\displaystyle{\Xi^{t}_{\pm}}$ preserves support in $\displaystyle{\mathbb{R}^N_{\pm}}$. Similarly one can defined $\displaystyle{\Lambda^t_{\pm}}$ adapted to the situation for a smooth domain. For details readers are refereed to \cite{hormander1963linear}, \cite{hormander1966seminar}, \cite{taylor2006pseudo}, \cite{taylor1991pseudodifferential} \cite{grubb2008distributions}. Here we mention some important homeomorphism results regarding this operators: for all $s>0$
	\begin{align*}
		& \Xi^{t}_{\pm}: H^s\left(\mathbb{R}^N\right) \xrightarrow{\sim} H^{s-t}\left(\mathbb{R}^N\right), \ \Lambda^{t}_{\pm}: H^s\left(\mathbb{R}^N\right) \xrightarrow{\sim} H^{s-t}\left(\mathbb{R}^N\right),\\
		& \Xi^{t}_{\pm}: \dot{H}^s\left( \overline{\mathbb{R}^{N}_{\pm}}\right) \xrightarrow{\sim}\dot{H}^{s-t}\left( \overline{\mathbb{R}^{N}_{\pm}}\right), \ r^{+}\Xi^{t}_{\pm} e^{+}: \overline{H}^s\left( {\mathbb{R}^{N}_{\pm}}\right) \xrightarrow{\sim}\overline{H}^{s-t}\left( {\mathbb{R}^{N}_{\pm}}\right),\\
		& \Lambda^{t}_{\pm}: \dot{H}^s\left( \overline{\Omega}\right) \xrightarrow{\sim}\dot{H}^{s-t}\left( \overline{\Omega}\right), \ r^{+}\Lambda^{t}_{\pm} e^{+}: \overline{H}^s\left( {\Omega}\right) \xrightarrow{\sim}\overline{H}^{s-t}\left( {\Omega}\right).
	\end{align*}
	We recall $\mu$-transmission spaces with Re $\mu>-1$ as introduced by H\"ormander \cite{hormander1963linear}, \cite{hormander1966seminar} and redefined in \cite{grubb2008distributions}:
	\begin{align*}
		H^{\mu(s)}\left(\overline{\mathbb{R}_{+}^N}\right):= &\Xi_{+}^{-\mu} e^{+}\overline{H}^{s-\mu}\left(\mathbb{R}_{+}^N\right), \quad s>\mu-\frac{1}{2}\\
		H^{\mu(s)}\left(\overline{\Omega}\right):= & \Lambda_{+}^{(-\mu)} e^{+} \overline{H}^{s-\mu}(\Omega), \ \ \quad s>\mu-\frac{1}{2}.
	\end{align*}
	The space $\displaystyle{{\mathcal{E}_{\mu}\left(\overline{\Omega}\right)}}$ is dense in $\displaystyle{H^{\mu(s)}\left(\overline{\Omega}\right)}$ i.e., [see \cite{hormander1963linear}, \cite{hormander1966seminar}, \cite{grubb2008distributions}]
	\[
	\cap_{s} H^{\mu(s)}\left(\overline{\Omega}\right) = \mathcal{E}_{\mu}\left(\overline{\Omega}\right).
	\]
	Also we have the following characterization [see \cite{grubb2008distributions}]:
	\begin{equation*}
		H^{\mu(s)}\left(\overline{\Omega}\right)
		\left\{
		\begin{aligned}
			&= \dot{H}^s \left(\overline{\Omega}\right) \ \ \ \mbox{if} \ |s-\mu|<\frac{1}{2},\\
			& \subset \dot{H}^s \left(\overline{\Omega}\right)+ e^{+}\tilde{d}^{\mu}\overline{H}^{s-\mu}(\Omega), \ \ \mbox{if}\ s>\mu+\frac{1}{2}, s-\mu-\frac{1}{2}\not \in \mathbb{N}.
		\end{aligned}
		\right .
	\end{equation*}
Some further characterization can be found in \cite{grubb2015fractional}, \cite{grubb2014local}. These conditions are stated in the lemma below:
\begin{Lem}\label{Further chracterization of mu transmission space}
	The following properties hold:
	\begin{itemize}
		\item[a:] For all $s>\mu-\frac 12$, $H^{\mu(s)}(\overline{\Omega})\subset \overline{H}^{\mu-\frac12}(\Omega)$ with continuous inclusion.
		\item[b:] 
		$H^{\mu(s)}(\overline{\Omega})=\dot{H}^{\mu+\frac 12-\epsilon}(\overline{\Omega})$, for all $s>\mu-\frac 12$ and $\epsilon>0$.
		\item[c:] $H^{\mu(s)}(\overline{\Omega})=\dot{H}^{s}(\overline{\Omega})$ if $\displaystyle{\mu\in\left( s-\frac12, s+\frac12\right)}$.
		\item[d:] Let, $q\in C_c^{\infty}(\Omega)$. The operator $\displaystyle{r^+\big((-\Delta)^s+q\big)}$ is a homeomorphism  from $H^{\mu(s)}(\overline{\Omega})$ onto $H^{\mu-2s}(\Omega)$.
		\item[e:] $\overline{H}^{s}(\Omega) \subset H^{\mu(s)}(\overline{\Omega}) \subset H^{s}_{loc}(\Omega)$ with continuous inclusions; i.e., multiplication by any $\chi\in C_c^{\infty}(\Omega)$ is bounded $H^{\mu(s)}(\overline{\Omega}) \to H^s(\Omega)$.
	\end{itemize}
\end{Lem}
	The domain of fractional  Laplacian with homogeneous Dirichlet boundary condition can be characterize as follows:
\begin{Lem}[\cite{claus2020realization}, \cite{grubb2018regularity}]\label{domain} 
	Let $\Omega\subset\mathbb{R}^N$ ($N\ge 1$) be a bounded open set with $C^{1,1}$ boundary and $a\in(0,1)$. Let $D((-\Delta)^a)$ be the domain of the operator $(-\Delta)^a$ with homogeneous Dirichlet boundary condition outside $\Omega$.  Then, the following assertions holds.
	\begin{itemize}
		\item[(a)] If $a\in(0,1/2)$, then $D((-\Delta)^s)= \dot{H}^{2a}(\Omega)$.
		\item[(b)] If $a=1/2$, then $D((-\Delta)^a)\subset H^{1-\varepsilon}(\Omega)\cap \dot{H}^a(\Omega)$ for all $\varepsilon\in(0,1)$.
		\item[(c)] If $a\in(1/2,1)$, then $D((-\Delta)^a)\subset H^{a+\frac 12}(\Omega)\cap \dot{H}^a(\Omega)$.
	\end{itemize}
\end{Lem}	
	We now state some  wellposedness and regularity results corresponding to \eqref{varying initial data} and \eqref{varying boundary data} in $\mu$-transmission spaces. These results are useful to obtain an integration by parts formula. Proof of these non-local parabolic wellposedness and regularity results can be found in \cite{grubb2023resolvents}, \cite{fernandez2016boundary}.
	\begin{Thm}\label{s-transmission regularity, initial data}
		Let $f\in \mathrm{L}^2(\Omega)$, then the unique solution $u_{f}$ to \eqref{varying initial data} has the following regularity:
		\[
		u_f\in C\left((0,T); H^{a(2a)}(\overline{\Omega})\right) \cap C^1((0,T); \mathrm{L}^2(\Omega))\cap C([0,T); \mathrm{L}^2(\Omega)).
		\]
	\end{Thm}
It follows from semigroup theory \cite{pazy2012semigroups}, \cite{brezis2011functional}.
	\begin{Thm}[\cite{grubb2023resolvents}]\label{s-transmission regularity, boundary data}
		Let $F\in \mathrm{L}^2\left((0,T); H^{a+\frac{1}{2}}(\partial\Omega)\right)\cap \overline{H}^1\left((0,T); H^{\epsilon}(\partial\Omega)\right)$, for some $\epsilon>0$ and let $F(0,\cdot)=0$. Then the solution $u_F$ to \eqref{varying boundary data} has the following regularity:
		\[
		u_F\in \mathrm{L}^2\left((0,T); H^{(a-1)(2a)}(\overline{\Omega})\right) \cap \overline{H}^1\left((0,T); \mathrm{L}^2(\Omega)\right).
		\]
	\end{Thm}
	 The theorem \ref{s-transmission regularity, initial data} is also true for the solution of the adjoint system corresponding to \eqref{varying initial data}, given by:
	\begin{equation}\label{adjoint system initial data}
		\left \{
		\begin{aligned}
			-\partial_t u_{f^*}+(-\Delta)^a u_{f^*}+ qu_{f^*}=&0 \qquad \qquad  \mbox{in}\ (0,T)\times\Omega\\
			u_{f^*}=&0 \qquad \qquad  \mbox{in} \ (0,T)\times\mathbb{R}^N\setminus\Omega \\
			u_{f^*}(T,\cdot)=&f^* \qquad \qquad  \mbox{in} \ \Omega,
		\end{aligned}
		\right .
	\end{equation}
	where $f^{*}\in\mathrm{L}^2(\Omega)$. It just comes from change of variable $t\to T-t$. We like to emphasize that Pohozaev identity \eqref{Phozaev identity} holds for $u_{f^*}$ also. Next we move onto defining traces of the transmission spaces.
	\subsection{Boundary values of $\mu$ transmission spaces, Dirichlet trace}
	The definitions and notations used in this sub-section are well explained in \cite{grubb2015fractional}. The author in \cite{grubb2015fractional}, explained that the $\mu$ transmission spaces $H^{\mu(s)}$ admits a special definition of $\mu$ boundary values, which characterize how the elements of the $\mu$ transmission spaces decays towards the boundary. Let $M$ be a positive integer. We recall the space $\mathcal{E}_{\mu}$ as defined in \eqref{density in transmission space}. Let us introduce the natural canonical mapping 
	\begin{align}\label{canonical map}
		\rho_{\mu,M}: \mathcal{E}_{\mu}\to \mathcal{E}_{\mu}/\mathcal{E}_{\mu+M},
	\end{align}
	where one can represent the space $\displaystyle{\frac{\mathcal{E}_{\mu}}{\mathcal{E}_{\mu+M}}}$ as the space of sections of a trivial bundle and introduce norms in it. Let us choose a smooth function $\tilde{d}$ on $\overline{\Omega}$ which is equal to the distance from the boundary sufficiently close to the boundary and remains positive in $\Omega$. We set
	\[
	I^{\mu}(x):= \frac{\tilde{d}^{\mu}(x)}{\Gamma(\mu+1)} \ \mbox{in}\ \overline{\Omega} \ \ \mbox{and}\ I^{\mu}\equiv 0 \ \mbox{in} \partial\Omega,
	\]
	when Re$(\mu)$>-1. This definition can be uniquely extended modulo $C_0^{\infty}(\Omega)$ to arbitrary values of $\mu$ so that $\partial_N I^{\mu}=I^{\mu-1}$, where $\partial_N$ is the derivative with respect to the $x_N$ variable. Note that fact that the line $x_N$ is perpendicular to the boundary $\partial\Omega$. Hence, it can be viewed as the differentiation along the geodesics perpendicular to the boundary. Our definition on $\mathcal{E}_{\mu}$ \eqref{density in transmission space}, clearly indicates that every class in $\displaystyle{\mathcal{E}_{\mu}/\mathcal{E}_{\mu+1}}$ contains an element of the form $I^{\mu}f$ where $f\in C^{\infty}(\overline{\Omega)}$. Furthermore, in the quotient space that element is equivalent to `$0$' if and only if $f=0$ on the boundary $\partial\mathbb{R}_+^N$. Repeated application of the previous fact yields the following representation of $u\in\mathcal{E}_{\mu}$:
	\[
	u= u_0 I^{\mu}+u_1I^{\mu+1}+\cdots+u_{M-1}I^{\mu+M-1}+v,
	\] 
	where the functions $u_i\in C^{\infty}(\overline{\Omega})$ and constant along the normal closed to the boundary $\partial\Omega$. So, the traces of $u$ denoted by the following notation
	\[
	\gamma_{\mu,j}u= u_j|_{\partial\Omega}.
	\]
	Note that the fact that closed to the boundary, $I^{\mu}$ is equal to $\displaystyle{\frac{d^{\mu}}{\Gamma(\mu+1)}}$. Furthermore, note that:
	\begin{align*}
		\gamma_{\mu,j}u=& \gamma_{\mu+j,0}u, \quad \mbox{when} \ u\in\mathcal{E}_{\mu+j},\\
		\gamma_{\mu,0}u=& \Gamma(\mu+1)\gamma_0 d^{-\mu}u, \quad \mbox{when} \ u\in\mathcal{E}_{\mu}, \ \mbox{with} \ Re(\mu)>-1.
	\end{align*}
	where $\gamma_0$ is the projection of a function on the boundary $\partial\Omega$. We can extend the boundary trace for the functions in $H^{\mu(s)}(\overline{\Omega})$. For a function
	$u\in H^{\mu(s)}(\overline{\Omega})$, the boundary term $\frac{u}{d^{\mu}}$ is well defined and lies in the space $H^{s-\mu-\frac12}(\partial\Omega)$ [see \cite{grubb2017fractional}, \cite{grubb2018green}]. The boundary trace $\frac{u}{d^{\mu}}$ for the function $u\in H^{\mu(s)}(\overline{\Omega})$, we call it as \emph{Dirichlet trace}. The following result can be found in \cite{grubb2015fractional}:
	\begin{Thm}[\cite{grubb2015fractional}]\label{dirichlet kernel in mu-transmission space}
		Let $s>$Re$(\mu)+M-\frac{1}{2}$. The canonical map $\rho_{\mu,M}$ in \eqref{canonical map} can be extended by continuity (still denoted by $\rho_{\mu,M}$):
		\[
		\rho_{\mu,M}: H^{\mu(s)}(\overline{\Omega})\to \prod_{0\leq j<M} H^{s-Re(\mu)-j-\frac12}(\partial\Omega).
		\]
		Furthermore, the map is surjective and with kernel $H^{(\mu+M)(s)}(\overline{\Omega})$. In other words, the canonical map $\rho_{\mu,M}$ defines a homeomorphism of $\displaystyle{H^{\mu(s)}(\overline{\Omega})/H^{(\mu+M)(s)}(\overline{\Omega})}$ onto $\displaystyle{\prod_{0\leq j<M} H^{s-Re(\mu)-j-\frac12}(\partial\Omega)}$.
	\end{Thm}
	The theorem implies that if a function $u\in H^{\mu(s)}(\overline{\Omega})$ has zero Dirichlet trace, with $s>Re(\mu)+M-\frac12$, then $u\in H^{(\mu+1)(s)}(\overline{\Omega})$.  Similarly, we can define the Neumann trace also for the transmission spaces $H^{(a-1)(2a)}(\overline{\Omega})$. 
	\newline
	With all these regularity results in hand we are going to state integration by parts formula: If $\displaystyle{u,v\in H^{(a-1)(s)}(\overline{\Omega})}$ with $\displaystyle{s>a+\frac{1}{2}}$, then
	\[
	\int_{\Omega}v(-\Delta)^a u -u(-\Delta)^a v = \Gamma(a)\Gamma(a+1) \int_{\partial\Omega} \partial_{\nu} \left(\frac{u}{d^{a-1}}\right) \frac{v}{d^{a-1}}- \partial_{\nu} \left(\frac{v}{d^{a-1}}\right) \frac{u}{d^{a-1}},
	\]
	where $\displaystyle{\partial_{\nu}}$ is the normal derivative at the boundary. Furthermore, if $u\in H^{(a-1)(s)}(\overline{\Omega})$ and $v\in H^{a(s)}(\overline{\Omega})$, then
	\begin{align}\label{integration by parts}
		\int_{\Omega}v(-\Delta)^a u -u(-\Delta)^a v = -\Gamma(a)\Gamma(a+1) \int_{\Omega} \frac{u}{d^{a-1}}\frac{v}{d^a}.
	\end{align}
	Here the boundary contributions are completely local. The formula was first proved for fractional Laplacian in \cite{abatangelo2013large} and later generalized for any classical pseudodifferential operator of order $2a$ in \cite{grubb2016integration}, \cite{Grubb2020ExactGF}. In articles \cite{grubb2016integration}, \cite{Grubb2020ExactGF}, author shows that the boundary terms $\displaystyle{\partial_{\nu}\left(\frac{u}{d^{a-1}} \right), \partial_{\nu}\left(\frac{v}{d^{a-1}} \right)\in H^{s-a-\frac{1}{2}}(\partial\Omega)}$  (in case of \eqref{integration by parts} $\displaystyle{\frac{u}{d^{a}}  \in H^{s-a-\frac{1}{2}}(\partial\Omega)}$ $)$. Hence if $s>a+\frac12$, then the $\mathrm{L}^2$ integral at the boundary makes sense.  We would like to note that for any $a\in(\frac12,1)$ and for any $t>0$, the integration by parts formula holds for the functions $u_f$ and $u_F$, which are the solutions to \eqref{varying initial data} and \eqref{varying boundary data}, respectively, where $f\in\mathrm{L}^2(\Omega)$ and $F\in\mathrm{L}^2((0,T)\times\partial\Omega)$. thanks to theorem \ref{s-transmission regularity, initial data} and theorem \ref{s-transmission regularity, boundary data}. 
	\newline
	We also mention a particular integration by parts formula which we will use throughout the article. Let $a\in(0,1)$ and $w_1,w_2\in H^{a(2a)}(\overline{\Omega})\cap H^a(\mathbb{R}^N)$ with $w_1\equiv w_2\equiv 0$ in $\displaystyle{\mathbb{R}^N\setminus\Omega}$, then
	\begin{align}\label{integration by parts Hs}
		\int_{\Omega} w_2 \, (-\Delta)^aw_1  \, \rm{d}x= \int_{\Omega} w_1 \, (-\Delta)^a w_2  \, \rm{d}x.
	\end{align}
	This can be established easily using linear translation.
	\newline
	Next, we proceed to establish our density results. We present two different proofs for each of these results, addressing both qualitative and quantitative aspects.

	\section{Density Results:}
This section we start with the density of the solution space corresponding to \eqref{varying initial data}. 

\vspace{.2cm}
	
\textbf{Proof of theorem \ref{Density result initial data}:}
	\begin{proof}
		We will provide two different proofs. One using integration by parts \eqref{integration by parts} and  another using a variational approach.
		\newline
		First proof: We select the time slice $t=\Tilde{T}$. Let the set $\displaystyle{\left\{u_f(\Tilde{T},\cdot), f\in C_c^{\infty}(\Omega)\right\}}$ is not dense in $\mathrm{L}^2(\Omega)$. Then by Hahn-Banach theorem there exists $f^*\in\mathrm{L}^2(\Omega)$ such that $\Vert f^*\Vert_{L^2(\Omega)}=1$ and 
		\begin{align}\label{inner product zero, initial data}
			\langle u_f(\Tilde{T},\cdot), f^* \rangle=0, \qquad \forall f\in C_c^{\infty}(\Omega).  
		\end{align}
		Let $u_{f^*}$ be the solution of the adjoint system \eqref{adjoint system initial data} with $f^*$ as initial data. Thanks to Theorem \ref{s-transmission regularity, initial data}, we have that
		\[
			u_{f^*}\in C\left((0,\tilde{T}); H^{a(2a)}(\overline{\Omega})\right) \cap C^1((0,\tilde{T}); \mathrm{L}^2(\Omega))\cap C([0,\tilde{T}]; \mathrm{L}^2(\Omega)).
		\]
		 We use the integration by parts formula \eqref{integration by parts Hs}
		\[
		\int_{\Omega} u_f(-\Delta)^a u_{f^*}=\int_{\Omega}u_{f^*}(-\Delta)^au_f 
		\]
		implies
		\[
		\int_{\Omega} \partial_t u_{f^*}u_f +u_{f^*}\partial_t u_f =0= \int_{\Omega}\partial_{t}\left(u_{f^*}u_f\right).
		\] 
		Integrating with respect to $t$ yields: 
		\[
		\int_{\Omega} u_{f^*}(\Tilde{T}-\epsilon,\cdot)u_f(\Tilde{T}-\epsilon,\cdot)=\int_{\Omega} u_{f^*}(0,\cdot)u_f(0,\cdot), \quad \forall\, \epsilon>0.
		\]
		Thanks to continuity with respect to time variable (Theorem \ref{s-transmission regularity, initial data}) and \eqref{inner product zero, initial data}, we deduce that
		\[
		u_{f^*}(0,\cdot)\equiv 0.
		\] 
		Since $f\in C_c^{\infty}(\Omega)$ is arbitrary. Furthermore, the solution can be represented in the following form [see \cite{chan2022singular}]:
		\[
		u_{f^*}(t,x)= \sum\limits_{n\in\mathbb{N}} e^{-\lambda_n(\Tilde{T}-t)} \phi_n(x)\langle \phi_n, f^{*}\rangle
		\]
		where for all $n\in\mathbb{N}$, $\phi_n$'s are the normalized eigenfunctions corresponding to \eqref{eigen function equation}. We have that $\lambda_1<\lambda_2\leq \dots\lambda_n\leq\dots+\infty$ and $\{\phi_n\}$ is an orthonormal basis of $\mathrm{L}^2(\Omega)$. Hence $u_{f^*}(0,x)\equiv0$ implies 
		\[
		e^{-\Tilde{T}\lambda_n}\langle \phi_n, f^{*}\rangle=0, \qquad \forall \ n\in\mathbb{N}.
		\]   
		This further implies $f^{*}\equiv0$, contradicts $\Vert f^{*}\Vert_{\mathrm{L}^2(\Omega)}=1$. Hence, for a given $\Tilde{T}>0$, the set $\{u_f(\Tilde{T},\cdot): f\in C_c^{\infty}(\Omega)\}$ is dense in $\mathrm{L}^2(\Omega)$.
		\newline
		Second proof: We now provide a quantitative proof based on variational approach. We prove that the solution set $\{u_f(\Tilde{T},\cdot): f\in C_c^{\infty}(\Omega)\}$ is dense in $\mathrm{L}^2(\Omega)$. 
		\newline
		Consider the functional: 
		\[
		\mathcal{J}_{\epsilon,a} (f^*) =\frac{1}{2} \int_{\Omega} \vert  u_{f^*}(0,\cdot)\vert^2 \, \rm{d}x+\epsilon \Vert f^{*}\Vert_{\mathrm{L}^2(\Omega)}-\int_{\Omega} h f^* \, \rm{d}x,
		\]
		where $\epsilon>0$ is a fix quantity and $h\in\mathrm{L}^2(\Omega)$ is a given function. Furthermore $u_{f^*}$ and $f^*$ are connected by \eqref{adjoint system initial data} with initial data $f^*$ at $T=\Tilde{T}$. We intend to show that the functional $\displaystyle{\mathcal{J}_{\epsilon,a}: \mathrm{L}^2(\Omega)\to \mathbb{R}\cup\{+\infty\}}$ is lower semicontinuous, strictly convex and coercive. 
		\begin{itemize}
			\item[(a).] Continuity: Let $f_n^{*}\to f^{*}$ in $\mathrm{L}^2(\Omega)$. The solution to \eqref{adjoint system initial data}, $u_{f^*_m}$ corresponding to the initial data $f_m^*$ can be expressed as
			\[
			u_{f^*_m}(t,\cdot)= \sum\limits_{n\in\mathbb{N}} e^{-\lambda_n(\Tilde{T}-t)} \phi_n(x)\langle \phi_n, f_m^*\rangle,
			\]
			where $\phi_n$ are the normalized eigenfunctions of \eqref{eigen function equation}. Hence 
			\[
			\Vert u_{f_m^*}(0,\cdot)-u_{f^*}(0,\cdot)\Vert^2_{\mathrm{L}^2(\Omega)}= \sum\limits_{n\in\mathbb{N}} e^{-2\Tilde{T}\lambda_n} | \langle \phi_n, f_m^*\rangle -\langle \phi_n, f^*\rangle|^2.  
			\]
			From the fact that $f_n^*\to f^*$ in $\mathrm{L}^2(\Omega)$, we have that 
			\[
			\displaystyle{\sum\limits_{m\in\mathbb{N}}| \langle \phi_n, f_m^*\rangle -\langle \phi_n, f^*\rangle|^2 \to 0}.
			\]
			Hence 
			\[
			\Vert u_{f_m^*}(0,\cdot)-u_{f^*}(0,\cdot)\Vert^2_{\mathrm{L}^2(\Omega)} \leq e^{-2\lambda_1\Tilde{T}}\sum\limits_{n\in\mathbb{N}}| \langle \phi_n, f_m^*\rangle -\langle \phi_n, f*\rangle|^2 \to 0. 
			\]
			This proves continuity of the functional $\mathcal{J}_{\epsilon,a}$.
			\item[(b).] Strictly Convexity: It follows from the strict convexity of the the square of the $\mathrm{L}^2$-norm.
			\item[(c).] Coercivity: Let $f_m^*\in\mathrm{L}^2(\Omega)$, for all $m\in\mathbb{N}$ and $\displaystyle{\Vert f_m^*\Vert_{\mathrm{L}^2(\Omega)}\to +\infty}$ as $m\to+\infty$. Consider the functions $\displaystyle{g_m^*= \frac{f_m^*}{\Vert f_m^*\Vert_{\mathrm{L}^2(\Omega)}}}$. Then $\displaystyle{\Vert g_m^*\Vert_{\mathrm{L}^2(\Omega)}=1}$ for all $m\in\mathbb{N}$. Let $u_{g^*_m}$ be the solution to \eqref{adjoint system initial data} with respect to the initial data $g_m^*$. Then $\displaystyle{u_{g^*_m}=\frac{u_{f^*_m}}{\Vert f_m^*\Vert_{\mathrm{L}^2(\Omega)}}}$. We divide the functional $\mathcal{J}_{\epsilon,a}$ by $\Vert f_m^*\Vert_{\mathrm{L}^2(\Omega)}$. It yields
			\[
			\frac{\mathcal{J}_{\epsilon,a}(f_m^*)}{\Vert f_m^*\Vert_{\mathrm{L}^2(\Omega)}}= \frac{1}{2} \Vert f_m^*\Vert_{\mathrm{L}^2(\Omega)} \int_{\Omega}  \vert u_{g_m^*}(0,\cdot)\vert^2 \, \rm{d}x+\epsilon-\int_{\Omega} hg_m^* \, \rm{d}x.
			\]
			There can be two different scenarios here. If $\displaystyle{\liminf\limits_{m\in\mathbb{N}}\int_{\Omega}  \vert u_{g_m^*}(0,\cdot)\vert^2>0}$, then 
			$\displaystyle{	\frac{\mathcal{J}_{\epsilon,a}(f_m^*)}{\Vert f_m^*\Vert_{\mathrm{L}^2(\Omega)}}}\to +\infty$ as $m\to+\infty$. The other case when there exists a subsequence (still indexed by $m$) such that $\displaystyle{\liminf\limits_{n\in\mathbb{N}}\int_{\Omega}  \vert u_{g_m^*}(0,\cdot)\vert^2=0}$. From the compactness of $\mathrm{L}^2$ space, we have that $g_m^* \rightharpoonup g^*$ weakly in $\mathrm{L}^2(\Omega)$. The functions $u_{g_m^*}$ can be expressed as 
			\[
			u_{g_m^*}(0,\cdot)= \sum\limits_{n\in\mathbb{N}} e^{-\Tilde{T}\lambda_n} \phi_n \langle \phi_n, g_m^*\rangle,
			\]
			where $\phi_n$'s are the eigenfunctions corresponding to \eqref{eigen function equation}. Let $u_{g^*}$ be the solution to \eqref{varying initial data} with initial data $g^*$. 
			\[
			u_{g^*}(0,\cdot)= \sum\limits_{n\in\mathbb{N}} e^{-\Tilde{T}\lambda_n} \phi_n \langle \phi_n, g^*\rangle.
			\]
			The lower semicontinuity of $\mathrm{L}^2$ norm implies $\displaystyle{\Vert g^*\Vert_{\mathrm{L}^2(\Omega)}\leq \Vert g^*_m\Vert_{\mathrm{L}^2(\Omega)}\leq 1}$. 
			We will show $u_{g^*_m}(0,\cdot)\to u_{g^*}(0,\cdot)$ converges in $\mathrm{L}^2(\Omega)$ strongly. 
			\begin{align*}
				\Vert u_{g_m^*}(0,\cdot)-u_{g^*}(0,\cdot)\Vert^2_{\mathrm{L}^2(\Omega)} =& \sum\limits_{n\in\mathbb{N}}e^{-2\Tilde{T}\lambda_n} |\langle \phi_n, g^*_m\rangle-\langle \phi_n, g^*\rangle|^2\\
				\leq & \sum\limits_{n=1}^{n_K}e^{-2\Tilde{T}\lambda_n} |\langle \phi_n, g^*_m\rangle-\langle \phi_n, g^*\rangle|^2 \\
				+& 2 \sum\limits_{n=n_K+1}^{\infty}e^{-2\Tilde{T}\lambda_n} + 2 \sum\limits_{n=n_K+1}^{\infty}e^{-2\Tilde{T}\lambda_n}\\
				\leq & \sum\limits_{n=1}^{n_K}e^{-2\Tilde{T}\lambda_n} |\langle \phi_n, g^*_m\rangle-\langle \phi_n, g^*\rangle|^2 \\
				+& 4 C_1(a) \sum\limits_{n=n_K+1}^{\infty} \frac{1}{|\lambda_k|^{\frac{N}{s}}}\\
				\leq & \sum\limits_{n=1}^{n_K}e^{-2\Tilde{T}\lambda_n} |\langle \phi_n, g^*_m\rangle-\langle \phi_n, g^*\rangle|^2 \\
				+& 4C_1(a)C_2(\Omega,N,a) \sum\limits_{n=n_K+1}^{\infty} \frac{1}{k^2}.
			\end{align*}
			The last inequality follows from one sided Weyl's law [see \cite{chan2022singular}, \cite{blumenthal1959asymptotic}]. Here $C_1(a)$, $C_2(\Omega,N,a)$ are two positive constant depends on the domain, the dimension and the fractional power `$a$'. We choose $n_K\in\mathbb{N}$  such that $\displaystyle{\sum\limits_{n=n_K+1}^{\infty} \frac{1}{k^2}\leq \delta}$ for a given $\delta>0$. Hence for a given $\delta>0$, there exists large $m_K\in\mathbb{N}$ such that for all $m\geq m_{K}$:
			\[
			\Vert u_{g_m^*}(0,\cdot)-u_{g^*}(0,\cdot)\Vert^2_{\mathrm{L}^2(\Omega)} \leq 2\delta.
			\]
			Hence we deduce $\displaystyle{u_{g_m^*}(0,\cdot)\to u_{g^*}(0,\cdot)}$ strongly in $\mathrm{L}^2(\Omega)$, which further implies
			\[
			u_{g^*}\equiv 0.
			\]
			Hence $\displaystyle{e^{-\Tilde{T}\lambda_n}\langle \phi_n, g^*\rangle=0}$ i.e., $\displaystyle{\langle \phi_n, g^*\rangle=0}$ for all $n\in\mathbb{N}$. It yields 
			\[
			g^*\equiv 0.
			\]
			As $g_n^*$ converges to $g^*$ weakly we have $\displaystyle{\int_{\Omega} hg_n^* \to 0}$. Hence $\displaystyle{	\frac{\mathcal{J}_{\epsilon,a}(f_n^*)}{\Vert f_n^*\Vert_{\mathrm{L}^2(\Omega)}}>\frac{\epsilon}{2}}$ for large $m$. Hence we deduce $\displaystyle{\mathcal{J}_{\epsilon,a}(f_n^*)\to +\infty}$ as $\displaystyle{\Vert f_n^*\Vert_{\mathrm{L}^2(\Omega)}\to+\infty}$. Hence $\mathcal{J}_{\epsilon,a}$ is coercive.
		\end{itemize}
		Continuity, strict convexity and coercivity helps us conclude  that $\mathcal{J}_{\epsilon,a}$ attains it's minimum for some $f_{min}^*\in \mathrm{L}^2(\Omega)$. We fix a function $g^{*}\in\mathrm{L}^2(\Omega)$ and let $u_{g^*}$ be the solution to \eqref{adjoint system initial data} corresponding to the initial data $g^{*}$. We have
		\[
		\mathcal{J}_{\epsilon,a} (f_{min}^*+\alpha g^*)-\mathcal{J}_{\epsilon,a} (f_{min}^*) \geq 0, \quad \forall \, \alpha\in\mathbb{R}.
		\]
		For $\alpha\geq0$, it yields
		\begin{align*}
			\frac{1}{2} \int_{\Omega} 2\alpha u_{f^*_{min}}(0,\cdot)&u_{g^*}(0,\cdot) \, \rm{d}x + \frac{1}{2} \int_{\Omega} \alpha^2 (u_{g^*}(0,\cdot))^2 \, \rm{d}x \\
			+&\epsilon\frac{\alpha^2\Vert g^*\Vert^2_{\mathrm{L}^2(\Omega)}+2\alpha \int_{\Omega} f^*_{min}g^* \, \rm{d}x}{\Vert f^*+\alpha g^*\Vert_{\mathrm{L}^2(\Omega)}+\Vert f_{min}^*\Vert_{\mathrm{L}^2(\Omega)}}-\alpha \int_{\Omega} h g^* \, \rm{d}x \geq 0. 
		\end{align*}
		We divide the whole expression by $\alpha$ and let $\alpha\to0\uparrow$. It yields 
		\[
		\int_{\Omega} u_{f^*_{min}}(0,\cdot)u_{g^*}(0,\cdot) \, \rm{d}x + \epsilon \frac{\int_{\Omega} f^*_{min}g^* \, \rm{d}x}{\Vert f_{min}^*\Vert_{\mathrm{L}^2(\Omega)}}- \int_{\Omega} h g^* \, \rm{d}x \geq 0.
		\]
		Employing Cauchy-Schwartz inequality, we obtain 
		\begin{align}\label{Minimizer inequality 1}
			\int_{\Omega} u_{f^*_{min}}(0,\cdot)u_{g^*}(0,\cdot) \, \rm{d}x - \int_{\Omega} h g^* \, \rm{d}x \geq -\epsilon \Vert g^*\Vert_{\mathrm{L}^2(\Omega)}.
		\end{align}
		Similarly for $\alpha\leq 0$, we obtain 
		\begin{align}\label{Minimizer inequality 2}
			\int_{\Omega} u_{f^*_{min}}(0,\cdot)u_{g^*}(0,\cdot) \, \rm{d}x - \int_{\Omega} h g^* \, \rm{d}x \leq \epsilon \Vert g^*(0,\cdot)\Vert_{\mathrm{L}^2(\Omega)}.
		\end{align}
		Together, we obtain
		\begin{align} \label{minimizer weak formulation}
			\left| \int_{\Omega} u_{f^*_{min}}(0,\cdot)u_{g^*}(0,\cdot) \, \rm{d}x - \int_{\Omega} h g^* \, \rm{d}x \right| \leq \epsilon \Vert g^*\Vert_{\mathrm{L}^2(\Omega)}.
		\end{align} 
		We choose $f=\eta u_{f^*_{min}}(0,\cdot)$ for the initial condition in \eqref{varying initial data}, where $\eta\in C_c^{\infty}(\Omega)$ to be choosen later. Let $f_1=(1-\eta)u_{f^*_{min}}(0,\cdot)$. 
		Let $u_f,u_{f_1}$ be the solutions to \eqref{varying initial data} with initial datum $f$ and $f_1$ respectively. Thanks to  \eqref{H4}, we can conclude $f\in C_c^{\infty}(\Omega)$. [see \cite{fernandez2016boundary}, \cite{grubb2023resolvents}]. Integration by parts \eqref{integration by parts Hs} yields that:
		\begin{align*}
			\int_{\Omega} fu_{g^*}(0,\cdot) \, \rm{d}x&=\int_{\Omega} u_{f}(\Tilde{T},\cdot) g^{*} \, \rm{d}x\\
			\int_{\Omega} f_1u_{g^*}(0,\cdot) \, \rm{d}x&=\int_{\Omega} u_{f_1}(\Tilde{T},\cdot) g^{*} \, \rm{d}x.
		\end{align*}
		Substituting this in \eqref{minimizer weak formulation}, we obtain
		\begin{align*}
			\left\vert \int_{\Omega} \big(u_{f}(\Tilde{T},\cdot)+u_{f_1}(\Tilde{T},\cdot)-h\big)g^{*} \, \rm{d}x \right \vert \leq \epsilon \Vert g^*\Vert_{\mathrm{L}^2(\Omega)}.
		\end{align*}
		Since $g^*\in\mathrm{L}^2(\Omega)$ is arbitrary, we obtain:
		\[
		\Vert u_f(\Tilde{T},\cdot)+ u_{f_1}(\Tilde{T},\cdot)-h\Vert_{\mathrm{L}^2(\Omega)} \leq \epsilon. 
		\]
		Triangle inequality yields:
		\begin{align}\label{triangle inequality initial data}
			\Vert u_f(\Tilde{T},\cdot) -h\Vert_{\mathrm{L}^2(\Omega)} \leq \epsilon +\Vert  u_{f_1}(\Tilde{T},\cdot)\Vert_{\mathrm{L}^2(\Omega)}.
		\end{align}
		It remains to choose $\eta\in C_c^{\infty}(\Omega)$ such that $\displaystyle{\Vert  u_{f_1}(\Tilde{T},\cdot)\Vert_{\mathrm{L}^2(\Omega)}\leq \epsilon}$. Thanks to \eqref{H0}, we have
		\[
		\sup\limits_{t\geq T}\Vert u_{f_1}(\Tilde{T},\cdot)\Vert_{C_c^a(\mathbb{R}^N)}\leq C_1(\Tilde{T})\Vert f_1\Vert_{\mathrm{L}^2(\Omega)} \leq  C_1(\Tilde{T})\Vert (1-\eta) u_{f^*_{min}}(0,\cdot)\Vert_{\mathrm{L}^2(\Omega)}.
		\]
		Again, thanks to \eqref{H0}, we have
		\begin{align}\label{sup u fstar min}
			\Vert u_{f^*_{min}}(0,\cdot)\Vert_{C^a_c(\mathbb{R}^N)}\leq C^*_1(\Tilde{T})\Vert f^*_{min}\Vert_{\mathrm{L}^2(\Omega)}.
		\end{align}
		If $f^*_{min}\equiv 0$, then thanks to \eqref{sup u fstar min}, we can choose arbitrary $\eta$. So Let $f^*_{min}\neq 0$, and $\Tilde{\Omega}$ is a domain compactly contained in $\Omega$ such that 
		\[
		\mbox{Volume}\left|\Omega\setminus\Tilde{\Omega}\right|\leq \epsilon^2C_1^{-2}(\Tilde{T})\left(C_1^*\right)^{-2}(\Tilde{T}) \Vert f^*_{min}\Vert^{-2}_{\mathrm{L}^2(\Omega)}.
		\]
		Employing H\"older inequality, we obtain
		\[
		\Vert u_{f_1}(\Tilde{T},\cdot) \Vert_{\mathrm{L}^2(\Omega)} \leq C_1(\Tilde{T})\sup\{|1-\eta|\} \sup \Vert u_{f_{min}^*}(0,\cdot)\Vert_{C^0_c(\mathbb{R}^N)}  \left(\mbox{Volume}\left|\Omega\setminus\Tilde{\Omega}\right|\right)^{\frac{1}{2}},
		\]
		where $\eta\in C_c^{\infty}(\Omega)$ such that $\eta \equiv 1$ in $\Tilde{\Omega}$. The choice of $\Tilde{\Omega}$ yields 
		\[
		\Vert u_{f_1}(\Tilde{T},\cdot)\Vert_{\mathrm{L}^2(\Omega)}\leq \epsilon.
		\]
		Hence for a given $\Tilde{T}>0$, the set $\{u_f(\Tilde{T},\cdot): f\in C_c^{\infty}(\Omega)\}$ is dense in $\mathrm{L}^2(\Omega)$.
	\end{proof}	
	Next we are interested in singular boundary data system \eqref{varying boundary data}. We like to analyze the density in $\mathrm{L}^2(\Omega)$ space on a particular time slice by varying singular boundary data. We describe two different ways to obtain the density of the solution space. One through observability estimate (see \cite{biccari2025boundary}) and the other through boundary unique continuation type result. However, the observability estimate only holds for small non-negative potential `$q$' with small growth. We start with an observability estimate. 
	\begin{Prop}\label{observability estimate}
		Let $a\in\left(\frac12, 1\right)$ and let $f^*\in \mathrm{L}^2(\Omega)$ and let $0\leq \epsilon< T$. Let $u_{f^*}$ be the solution to \eqref{adjoint system initial data} with initial data $f^*$. Let $\displaystyle{\theta< \frac{1}{2}\left(1+C_{HS}\left(\frac N2+R\right)\right)^{-1}}$, where $C_{HS}$ is the Hardy-Sobolev constant on the domain $\Omega$ and $R:=\max\{\Vert x\Vert, x\in\Omega\}$. Let the potential $q$ is non-negative and $|q(x)|,|\nabla q(x)|\leq \theta$ for all $x\in\Omega$. Then the following observability estimate holds 
		\[
		\mathcal{A}\int_{0}^{T-\epsilon} \int_{\partial\Omega}\left(\frac{u_{f^*}}{d^a}\right)^2(x\cdot\nu) \, \rm{d}\sigma \, \rm{d}s \geq \Vert u_{f^*}(0)\Vert_{\mathrm{L}^2(\Omega)},  
		\]
		where, the positive constant $\mathcal{A}$, depends only on the domain, the dimension, the time factor $T-\epsilon$ and the exponent `$a$'.
	\end{Prop}
\begin{proof}
	The proof is exactly similar to the proof of observability estimate in \cite{biccari2025boundary}. The only difference is we have a small potential here, which was not present in \cite{biccari2025boundary}. The idea is to obtain a control function that renders the solution null controllable within a specific finite-dimensional subspace,
	\begin{align}\label{finite projective space}
	\mathcal{H}_{J}:=span\{\phi_1,\cdots,\phi_J\},
	\end{align}
	and then use this control as a test function in the adjoint system with the given initial data \eqref{adjoint system initial data}. Here $\{\phi_i\}\Big|_{i\in\mathbb{N}}$ is the eigenfunctions of the operator $(-\Delta)^a+q$, which generates an orthonormal basis of $\mathrm{L}^2(\Omega)$. In \cite{biccari2025boundary}, the authors achieved this by making a detour through the wave equation, establishing an observability estimate for its energy, specifically for solutions lying in the finite-dimensional subspace $\mathcal{H}_J$. Although in \cite{biccari2025boundary} authors don't have a potential `$q$'. We claim that even if we have a small potential we can obtain similar result. So, in order to prove Proposition \ref{observability estimate}, we make a little detour to the wave equation with a small potential and will obtain a observability estimate corresponding to it's energy, when the solutions are coming  from the space $\mathcal{H}_J$. We move the proof in Appendix-\ref{observability wave}. 
\end{proof}
The above Proposition enables us to prove  the density of the solution space corresponding to the singular boundary data \eqref{varying boundary data}.
	
\vspace{.2cm}
	
\textbf{Proof of theorem \ref{Quantitative Density result boundary data}:}
	\begin{proof}
		Again we  present two proofs here. One using integration by parts \eqref{integration by parts} and another by a variational approach. Due to technical reason, for the variational approach, we will assume the potential is small, non-negative and it's space growth is also small.
		\newline 
		First proof: Let the set $\displaystyle{\{u_F(\Tilde{T},\cdot), \, F\in C_c^{\infty}((0,T)\times\partial\Omega)\}}$ is not dense in $\mathrm{L}^2(\Omega)$. Then, by Hahn-Banach theorem, there exists $g\neq0$ in the normal of the closure of the solution set $\displaystyle{\{u_F(\Tilde{T},\cdot), \, F\in C_c^{\infty}((0,T)\times\partial\Omega)\}}$ such that $\Vert g\Vert_{\mathrm{L}^2(\Omega)}=1$ and
		\begin{align}\label{inner product zero, boundary data}
			\langle u_F(\Tilde{T},\cdot), g \rangle=0, \qquad \forall F\in C_c^{\infty}((0,T)\times\partial\Omega).
		\end{align}
		Let $u_{g}$ be the solution to the adjoint system \eqref{adjoint system initial data} with $g$ as the initial data at the time $T=\Tilde{T}$. Thanks to Theorem \ref{s-transmission regularity, initial data}, we have that
\[
u_{g}\in C\left((0,\tilde{T}); H^{a(2a)}(\overline{\Omega})\right) \cap C^1((0,\tilde{T}); \mathrm{L}^2(\Omega))\cap C([0,\tilde{T}]; \mathrm{L}^2(\Omega)).
\]
 We apply integration by parts formula \eqref{integration by parts}:
		\[
		\int_{\Omega} u_{g}(-\Delta)^a u_{F}-\int_{\Omega}u_{F}(-\Delta)^au_{g} =-\Gamma(a)\Gamma(a-1)\int_{\partial\Omega} \frac{u_{F}}{d^{a-1}}\frac{u_{g}}{d^a}.
		\]
		From equation \eqref{varying boundary data}, we obtain
		\[
		\int_{\Omega} u_F\partial_t u_{g} +u_{g}\partial_t u_F = \int_{\Omega}\partial_{t}\left(u_{g}u_F\right)=-\Gamma(a)\Gamma(a-1)\int_{\partial\Omega} \frac{u_{F}}{d^{a-1}}\frac{u_{g}}{d^a} .
		\] 
		Integrating with respect to $t$ and continuity in time variable yields: 
		\[
		\int_{\Omega} u_{g}(\Tilde{T},\cdot)u_F(\Tilde{T},\cdot) \, \rm{d}x=-\Gamma(a)\Gamma(a-1)\int_{0}^{\Tilde{T}}\int_{\partial\Omega} \frac{u_{F}}{d^{a-1}}\frac{u_{g}}{d^a}\, \rm{d}x \, \rm{d}s.
		\]
		Thanks to \eqref{inner product zero, boundary data} and since $F\in C_c^{\infty}((0,T)\times\partial\Omega)$ is arbitrary, we deduce that
		\[
		\frac{u_{g}}{d^a}\equiv 0.
		\] 
		Proposition \ref{observability estimate} yields
		\[
		 g\equiv 0 \ \ \ \mbox{a contradiction}.
		\]
		Next we present a proof via variational approach. It is more quantitative in nature. We prove that the solution set $\displaystyle{\{u_F(\Tilde{T},\cdot), \, F\in \mathrm{L}^2((0,T)\times\partial\Omega), F(0,\cdot)=0\}}$ is dense in $\mathrm{L}^2(\Omega)$. The theorem then follows because of the space $C_c^{\infty}((0,T)\times\Omega)$ is dense in $\mathrm{L}^2((0,T)\times\Omega)$ and \cite{biccari2025boundary}[Proposition A.13]. In order to apply observability estimate we assume the potential is non-negative and small as described in Proposition \ref{observability estimate}.
		\newline
		Second proof: Consider the functional 
		\[
		\mathcal{J}_{\epsilon,a} (F^*) =\frac{1}{2} \int_{0}^{\Tilde{T}}\int_{\partial\Omega} \eta^2_{\Tilde{T}}\left\vert  \frac{u_{F^*}}{d^a}\right\vert^2 \, \rm{d}x \, \rm{d}t+\epsilon \Vert {F^{*}}\Vert_{\mathrm{L}^2(\Omega)}-\int_{\Omega} h {F^{*}} \, \rm{d}x,
		\]
		where $\epsilon>0$ is a fix quantity, $h\in\mathrm{L}^2(\Omega)$ is a given function and $\eta_{\Tilde{T}}\in C_c^{\infty}((0,\Tilde{T}))$ such that $\eta_{\Tilde{T}}\equiv 1$ in $[T_1, \tilde{T}]\subset(0,\Tilde{T}]$. Furthermore $u_{F^*}$ and $F^*$ are connected by \eqref{adjoint system initial data} with $F^*$ as the initial data at $T=\Tilde{T}$. We intend to show the functional $\displaystyle{\mathcal{J}_{\epsilon,a}: \mathrm{L}^2(\Omega)\to \mathbb{R}\cup\{+\infty\}}$ is lower semicontinuous, strictly convex and coercive. Just like the case of theorem \ref{Density result initial data} continuity of $\mathcal{J}_{\epsilon,a}$ follows from the eigenvalue representation and \eqref{H1}. Here \eqref{H1} provides the boundary continuity. Strict convexity follows from the strict convexity of the $\mathrm{L}^2$-norm and oberservability estimate Proposition \ref{observability estimate}. We only prove coercivity of the functional $\mathcal{J}_{\epsilon,a}$.
		\begin{itemize}
			\item[(a).] Coercivity: Let $F_m^*\in\mathrm{L}^2(\Omega)$ for all $m\in\mathbb{N}$ and $\displaystyle{\Vert F_m^*\Vert_{\mathrm{L}^2(\Omega)}\to +\infty}$ as $m\to+\infty$. Consider the functions $\displaystyle{G_m^*= \frac{F_m^*}{\Vert F_m^*\Vert_{\mathrm{L}^2(\Omega)}}}$. Then $\displaystyle{\Vert G_m^*\Vert_{\mathrm{L}^2(\Omega)}=1}$ for all $m\in\mathbb{N}$. Let $u_{G^*_m}$ be the solution to \eqref{adjoint system initial data} with respect to the initial data $G_m^*$. Then $\displaystyle{u_{G^*_m}=\frac{u_{F^*_m}}{\Vert F_m^*\Vert_{\mathrm{L}^2(\Omega)}}}$. We divide the functional $\mathcal{J}_{\epsilon,a}$ by $\Vert F_m^*\Vert_{\mathrm{L}^2(\Omega)}$. It yields 
			\begin{align*}
				\frac{\mathcal{J}_{\epsilon,a}(F^*_m)}{\Vert F_m^*\Vert_{\mathrm{L}^2(\Omega)}}= \frac{1}{2} \Vert F_m^*\Vert_{\mathrm{L}^2(\Omega)} \int_0^{\Tilde{T}} \int_{\partial\Omega} \eta_{\Tilde{T}}^2 \left| \frac{u_{G_m^{*}}}{d^a}\right|^2 \, \rm{d}x \, \rm{d}t
				+ \epsilon - \int_{\Omega} h {G_m^*} \, \rm{d}x.
			\end{align*}
			There can be two different scenarios. One  is
			\[
			\liminf_{m\in\mathbb{N}}\int_0^{\Tilde{T}} \int_{\partial\Omega} \eta_{\Tilde{T}}^2 \left| \frac{u_{G_m^{*}}}{d^a}\right|^2 \, \rm{d}x \, \rm{d}t>0 \implies \mathcal{J}_{\epsilon,a}\to+\infty.
			\]
			The second scenario is $\displaystyle{\liminf_{m\in\mathbb{N}}\int_0^{\Tilde{T}} \int_{\partial\Omega} \eta_{\Tilde{T}}^2 \left| \frac{u_{G_m^{*}}}{d^s}\right|^2 \, \rm{d}x \, \rm{d}t=0}$. Thanks to the observability estimate Proposition \ref{observability estimate}, we have that  $\displaystyle{\Vert u_{G_m^{*}}(t)\Vert_{\mathrm{L}^2(\Omega)}\to 0}$ as $m\to+\infty$, for all $t\in[0,\Tilde{T}]$. Here the end points are also included because of continuity in time variable. We can represent the solution $u_{G_m^*}$ as
			\[
			u_{G_m^*}(t,x)= \sum\limits_{n\in\mathbb{N}} e^{-\lambda_n(\Tilde{T}-t)} \langle\phi_n, G_m^*\rangle\phi_n(x).
			\]
			Hence $\displaystyle{\langle\phi_n, G_m^*\rangle \to 0}$ as $m\to+\infty$. As $G_m^*$ is uniformly bounded in $\mathrm{L}^2$, there exists $G^*\in \mathrm{L}^2(\Omega)$ such that $G_m^* \rightharpoonup G^*$ weakly in $\mathrm{L}^2(\Omega)$. This further implies
			\[
			\langle\phi_n, G_m^*\rangle \to \langle\phi_n, G^*\rangle \to 0 \ \ \mbox{as}\ m\to +\infty.
			\]
			Hence $G^*\equiv 0$ and we have that 
			\[
			\int_{\Omega} h G_m^* \to 0 \ \ \mbox{as} \ m\to +\infty.
			\]
			Hence $\mathcal{J}_{\epsilon,a}(F_m^*)\to +\infty$ as $\displaystyle{\Vert F_m^*\Vert_{\mathrm{L}^2(\Omega)}\to+\infty}$ i.e., $\mathcal{J}_{\epsilon,a}$ is coercive. 
		\end{itemize}
		Continuity, strict convexity and coercivity helps us conclude  that $\mathcal{J}_{\epsilon,a}$ attains it's minimum for some $F_{min}^*\in \mathrm{L}^2(\Omega)$. We fix a function $G^{*}\in\mathrm{L}^2(\Omega)$ and let $u_{G^*}$ be the solution to \eqref{adjoint system initial data} corresponding to the initial data $G^{*}$. We have
		\[
		\mathcal{J}_{\epsilon,a} (F_{min}^*+\alpha G^*)-\mathcal{J}_{\epsilon,a} (F_{min}^*) \geq 0, \quad \forall \, \alpha\in\mathbb{R}.
		\]
		For $\alpha\geq0$, it yields
		\begin{align*}
			\frac{1}{2} &\int_0^{\Tilde{T}} \int_{\partial\Omega} 2\alpha \eta_{\Tilde{T}}^2  \frac{u_{F^*_{min}}}{d^a}\frac{u_{G^*}}{d^a} \, \rm{d}x \, \rm{d}t + \frac{1}{2} \int_0^{\Tilde{T}}\int_{\partial\Omega} \alpha^2 \eta_{\Tilde{T}}^2 \frac{(u_{G^*})^2}{d^{2a}} \, \rm{d}x \, \rm{d}t \\
			+&\epsilon\frac{\alpha^2\Vert G^*\Vert^2_{\mathrm{L}^2(\Omega)}+2\alpha \int_{\Omega} F^*_{min}{G^*} \, \rm{d}x}{\Vert F^*_{min}+\alpha G^*\Vert_{\mathrm{L}^2(\Omega)}+\Vert F_{min}^*\Vert_{\mathrm{L}^2(\Omega)}}-\alpha \int_{\Omega} h G^* \, \rm{d}x \geq 0. 
		\end{align*}
		We divide the whole expression by $\alpha$ and let $\alpha\to0\uparrow$. It yields 
		\begin{align*}
			\int_0^{\Tilde{T}} \int_{\partial\Omega} \eta_{\Tilde{T}}^2    \frac{u_{F^*_{min}}}{d^a}\frac{u_{G^*}}{d^a} \, \rm{d}x \, \rm{d}t 
			+ \epsilon \frac{\int_{\Omega} F^*_{min}G^* \, \rm{d}x}{\Vert F_{min}^*\Vert_{\mathrm{L}^2(\Omega)}}- \int_{\Omega} h G^* \, \rm{d}x \geq 0.
		\end{align*}
		Employing Cauchy-Schwartz inequality, we obtain 
		\begin{align*}
			\int_0^{\Tilde{T}} \int_{\partial\Omega}  \eta_{\Tilde{T}}^2  \frac{u_{F^*_{min}}}{d^a}\frac{u_{G^*}}{d^a} \, \rm{d}x \, \rm{d}t - \int_{\Omega} h G^* \, \rm{d}x \geq -\epsilon \Vert G^*\Vert_{\mathrm{L}^2(\Omega)}.
		\end{align*}
		Similarly for $\alpha\leq 0$, we obtain 
		\begin{align*}
			\int_0^{\Tilde{T}} \int_{\partial\Omega}    \eta_{\Tilde{T}}^2\frac{u_{F^*_{min}}}{d^a}\frac{u_{G^*}}{d^a} \, \rm{d}x \, \rm{d}t - \int_{\Omega} h G^* \, \rm{d}x \leq \epsilon \Vert G^*\Vert_{\mathrm{L}^2(\Omega)}.
		\end{align*}
		Together, we obtain:
		\begin{align} \label{minimizer weak formulation boundary data}
			\left|\int_0^{\Tilde{T}} \int_{\partial\Omega}    \eta_{\Tilde{T}}^2\frac{u_{F^*_{min}}}{d^a}\frac{u_{G^*}}{d^a} \, \rm{d}x \, \rm{d}t - \int_{\Omega} h G^* \, \rm{d}x \right|\leq \epsilon \Vert G^*\Vert_{\mathrm{L}^2(\Omega)}.
		\end{align} 
		We choose $\displaystyle{F=\eta_{\Tilde{T}}^2 \frac{u_{F^*_{min}}}{d^a}}$ for the singular boundary data in \eqref{varying boundary data}.
		Let $u_F$ be the solutions to \eqref{varying boundary data} with the singular boundary data $F$. Integration by parts \eqref{integration by parts} yields that:
		\begin{align*}
			\int_{\Omega} u_F(\Tilde{T},\cdot)G^*(\cdot) \, \rm{d}x=-\Gamma(a)\Gamma(a-1)\int_{0}^{\Tilde{T}}\int_{\partial\Omega} \frac{u_{F}}{d^{a-1}}\frac{u_{G^*}}{d^a}\, \rm{d}x \, \rm{d}s.
		\end{align*}
		Substituting this in \eqref{minimizer weak formulation boundary data}, we obtain
		\begin{align*}
			\left\vert \int_{\Omega} \big(u_{F}(\Tilde{T},\cdot)-h\big)G^{*} \, \rm{d}x \right \vert \leq \epsilon \Vert G^*\Vert_{\mathrm{L}^2(\Omega)}.
		\end{align*}
		Since $G^*\in\mathrm{L}^2(\Omega)$ is arbitrary, we obtain:
		\[
		\Vert u_F(\Tilde{T},\cdot)-h\Vert_{\mathrm{L}^2(\Omega)} \leq \epsilon. 
		\]
\end{proof}
	
Now we  move onto the proofs of the Calder\'on type inverse problems for the non-local  heat equations. We recover the potentials by analyzing the solution at a particular time slice for \eqref{varying initial data} and by analyzing the solution on the boundary for \eqref{varying boundary data}.

	\section{Calder\'on type Inverse Problems:}

We start with establishing the equivalency of the two Calder\'on type inverse problems as described in Theorem \ref{inverse result initial data} and Theorem \ref{inverse result boundary data} respectively. 	

\vspace{.2cm}
	
\textbf{Proof of theorem \refeq{inverse result initial data}:}
	\begin{proof}
		We show that proving Theorem \ref{inverse result initial data} is same as proving Theorem \ref{inverse result boundary data}. Let $u_1$, $u_2$ be the solution of \eqref{varying initial data inverse result}. Let for $i=1,2$, $w_i$ satisfies
		\begin{equation*}
			\left\{
			\begin{aligned}
				-\partial_t w_{i}+(-\Delta)^a w_i+ q_iw_i=&0 \qquad \qquad  \mbox{in}\ (0,T)\times\Omega\\
				w_i=&0 \qquad \qquad  \mbox{in} \ (0,T)\times\mathbb{R}^N\setminus\overline{\Omega}\\
				\lim\limits_{\substack{x\to y\\ y\in\partial\Omega}}\frac{w_i}{d^{a-1}}(y)=& F \qquad \qquad \mbox{on}\  (0,T)\times\partial\Omega\\
				w_i(T,\cdot)=&0 \qquad \qquad  \mbox{in} \ \Omega,
			\end{aligned}
			\right .
		\end{equation*}
		where $F\in C_c^{\infty}((0,T)\times\partial\Omega)$. Since $\displaystyle{w_i\in \mathrm{L}^2\left((0,T); H^{(a-1)(2a)}(\overline{\Omega})\right)} $  for $i=1,2$ [see theorem \ref{s-transmission regularity, boundary data}, theorem \ref{s-transmission regularity, initial data}], integration by parts \eqref{integration by parts} yields
		\[
		\int_{\Omega} u_{i}(-\Delta)^a w_{i}-\int_{\Omega}w_{i}(-\Delta)^a u_{i} =-\Gamma(a)\Gamma(a-1)\int_{\partial\Omega} \frac{w_{i}}{d^{a-1}}\frac{u_{i}}{d^a}, \quad i=1,2,
		\]
		where, for $i=1,2$, $u_i$'s are the solutions of \eqref{varying initial data inverse result}. Replacing $(-\Delta)^a u_i$ and $(-\Delta)^a w_i$ using the equations yields:
		\[
		\int_{\Omega} w_{i}\partial_t u_{i} +w_{i}\partial_t u_i = \int_{\Omega}\partial_{t}\left(u_{i}w_i\right)=-\Gamma(a)\Gamma(a-1)\int_{\partial\Omega} \frac{w_i}{d^{a-1}}\frac{u_i}{d^a}, \quad i=1,2 .
		\] 
		Integrating with respect to $t$ yields: 
		\[
		\int_{\Omega} w_{i}(0,\cdot)u_i(0,\cdot) \, \rm{d}x=-\Gamma(a)\Gamma(a-1)\int_{0}^{T}\int_{\partial\Omega} \frac{w_{i}}{d^{a-1}}\frac{u_{i}}{d^a}\, \rm{d}x \, \rm{d}s, \quad i=1,2.
		\]
		Subtracting we obtain
		\[
		\int_{\Omega} f(w_1(0,\cdot)-w_2(0,\cdot)) \, \rm{d}x=0,
		\]
		as we have $\displaystyle{\frac{u_1}{d^a}=\frac{u_2}{d^a}}$ for all $f\in C_c^{\infty}(\Omega)$. Hence we obtain
		\[
		w_1(0,\cdot)\equiv w_2(0,\cdot).
		\]
		Our goal is to obtain $q_1\equiv q_2$. So our problem boils down to the following problem: Let $w_i$, for $i=1,2$ satisfy
		\begin{equation*}
			\left\{
			\begin{aligned}
				-\partial_t w_{i}+(-\Delta)^a w_i+ q_iw_i=&0 \qquad \qquad  \mbox{in}\ (0,T)\times\Omega\\
				w_i=&0 \qquad \qquad  \mbox{in} \ (0,T)\times\mathbb{R}^N\setminus\overline{\Omega}\\
				\lim\limits_{\substack{x\to y\\ y\in\partial\Omega}}\frac{w_i}{d^{a-1}}(y)=& F \qquad \qquad \mbox{on}\  (0,T)\times\partial\Omega\\
				w_i(T,\cdot)=&0 \qquad \qquad  \mbox{in} \ \Omega,
			\end{aligned}
			\right .
		\end{equation*}
		where $F\in C_c^{\infty}((0,T)\times\partial\Omega)$. Let 
		\[
		w_1(0,\cdot)\equiv w_2(0,\cdot), \quad \forall \, F\in C_c^{\infty}((0,T)\times\partial\Omega).
		\]	
		Then 
		\[
		q_1\equiv q_2.
		\]
		The change of variable $t \to T-t$ transforms the problem into the one described in Theorem \eqref{inverse result boundary data}. Hence it is enough to prove theorem \eqref{inverse result boundary data}.
	\end{proof}

\textbf{Proof of theorem \eqref{inverse result boundary data}:}
\begin{proof} 
	Let $v=w_1-w_2$ satisfy
\begin{equation*}\label{difference equation}
	\left \{
	\begin{aligned}
		\partial_t v+(-\Delta)^av+(q_1w_1-q_2w_2)=& 0 \qquad \mbox{in} \ (0,T)\times\Omega\\
		v= & 0 \qquad \mbox{in} \ (0,T)\times\mathbb{R}^N\setminus\Omega\\
		v(0,\cdot)=& 0 \qquad \mbox{in} \ \{0\}\times\Omega. 
	\end{aligned}
	\right .
\end{equation*}	
Let $V$ satisfy
\begin{equation*}\label{Auxiliary test equation}
	\left \{
	\begin{aligned}
		\partial_t V+(-\Delta)^aV=& 0 \qquad \mbox{in} \ (0,T)\times\Omega\\
		V= & 0 \qquad \mbox{in} \ (0,T)\times\mathbb{R}^N\setminus\overline{\Omega}\\
		\frac{V}{d^{a-1}}=& F \qquad \mbox{on} \ (0,T)\times\partial\Omega\\
		V(0,\cdot)=& 0 \qquad \mbox{in} \ \{0\}\times\Omega, 
	\end{aligned}
	\right .
\end{equation*}	
where $F(0,\cdot)=0$ and $F\in \mathrm{L}^2((0,T); H^{a+\frac{1}{2}}(\partial\Omega))$. Hence Thanks to theorem \ref{s-transmission regularity, boundary data} \cite{grubb2023resolvents}, we have that 
\[
V \in \mathrm{L}^2\left((0,T); H^{(a-1)(2a)}(\overline{\Omega})\right) \cap \overline{H}^1\left((0,T); \mathrm{L}^2(\Omega)\right).
\] 
Integration by parts \eqref{integration by parts} yields
\begin{align*}
	\int_{\Omega} V(s,x)(-\Delta)^a v(t,x) \, \rm{d}x-&\int_{\Omega} v(t,x)(-\Delta)^a V(t,x) \, \rm{d}x\\
	&=
	-\Gamma(a)\Gamma(a+1) \int_{\partial\Omega}  \frac{V}{d^{a-1}}(t,\sigma)\frac{v}{d^{a}}(s,\sigma) \, \rm{d}\sigma=0.  
\end{align*}
Integrating with respect to variable `$t$' yields
\[
\int_0^T \int_{\Omega} V(s,x)(-\Delta)^a v(t,x)- v(t,x)(-\Delta)^a V(t,x) \, \rm{d}x \, \rm{d}t=0
\]
Equation \eqref{difference equation} and equation \eqref{Auxiliary test equation} yields
\[
\int_{0}^T \int_{\Omega} \left( -\partial_t v(t,x) -(q_1w_1-q_2w_2)(t,x)\right) V(s,x) \, \rm{d}x\, \rm{d}t= \int_{0}^T \int_{\Omega} \left( -\partial_s V(s,x)\right) v(t,x) \, \rm{d}x \, \rm{d}t. 
\]
Simplifying, it implies
\begin{align}\label{Neumann inverse problem parabolic: intermedite 1}
	-\int_{\Omega} v(T,x) V(s,x) \, \rm{d}x +& \int_{0}^{T}\int_{\Omega} v(t,x) \partial_s V(s,x) \, \rm{d}x\, \rm{d}t\\
	=& \int_{0}^T \int_{\Omega} (q_1w_1-q_2w_2)(t,x) V(s,x)\, \rm{d}x\, \rm{d}t \quad \mbox{for all} \ s>0. \nonumber
\end{align}
Consider the function $w:(0,+\infty)\times(0,+\infty)\to\mathbb{R}$, denoted by 
\[
w(t,s)=\int_{\Omega} v(t,x)V(s,x).
\]
We will show `$w$' satisfy a particular transport equation. We compute the term $\partial_t w-\partial_s w$.
\begin{align*}
	\partial_t w-\partial_s w = \int_{\Omega} \partial_t v(t,x) & V(s,x) \, \rm{d}x -\int_{\Omega} v(t,x) \partial_s V(s,x) \, \rm{d}x\\
	= - \int_{\Omega} V(s,x)&(-\Delta)^a v(t,x) \, \rm{d}x+ \int_{\Omega} v(t,x)(-\Delta)^a V(s,x) \, \rm{d}x\\
	-& \int_{\Omega} (q_1w_1-q_2w_2)(t,x)V(s,x)\, \rm{d}x
\end{align*} 
Employing integration by parts formula \eqref{integration by parts}, we obtain
\begin{align*}
	\partial_t w-\partial_s w=   -\int_{\Omega} & (q_1w_1-q_2w_2)(t,x)V(s,x)\, \rm{d}x\\
	+ & \Gamma(a)\Gamma(a+1) \int_{\partial\Omega} \frac{V}{d^{a-1}}(s,\sigma)\frac{v}{d^{a}}(t,\sigma)\, \rm{d}\sigma.
\end{align*}
As $\displaystyle{\frac{v}{d^{a}}\Big|_{\partial\Omega}=0}$, we obtain
\[
\partial_t w-\partial_s w=   -\int_{\Omega} (q_1w_1-q_2w_2)(t,x)V(s,x)\, \rm{d}x.
\]
Hence `$w$' satisfy the following transport equation 
\begin{equation*}
	\left \{
	\begin{aligned}
		\partial_t w- \partial_s w= & -\int_{\Omega} (q_1w_1-q_2w_2)(t,x)V(s,x)\, \rm{d}x \qquad \mbox{in}\ (0,+\infty)\times(0,+\infty)\\
		w(0,s)=& 0 \qquad \qquad \qquad \qquad\qquad \qquad \qquad \ \qquad \mbox{on}\ \{0\}\times(0,+\infty)\\
		w(t,0)=& 0 \qquad \qquad \qquad \qquad\qquad \qquad \qquad \ \qquad \mbox{on}\ (0,+\infty)\times\{0\}.
	\end{aligned}
	\right .
\end{equation*}
We perform a change of variable `$t-s\to\xi$' and `$s\to s$'. We can rewrite the equation corresponding to `$w$' as 
\[
\partial_{\xi} w= -\int_{\Omega} (q_1w_1-q_2w_2)(\xi+s,x)V(s,x)\, \rm{d}x.
\]
Hence for any $s\geq 0$, `$w$' satisfies the following ordinary differential equation: 
\begin{equation*}
	\left\{
	\begin{aligned}
		\partial_{\xi} w=& -\int_{\Omega} (q_1w_1-q_2w_2)(\xi+s,x)V(s,x)\, \rm{d}x \qquad \mbox{in} \ (-s,+\infty)\\
		w(-s)=& 0.
	\end{aligned}
	\right .
\end{equation*}
This implies
\[
w(\xi',s)= -\int_{-s}^{\xi'}\int_{\Omega} (q_1w_1-q_2w_2)(\xi+s,x)V(s,x)\, \rm{d}x\, \rm{d}t.
\]
Let $\xi'+s=T$, then change of variable yields
\begin{align*}
	w(T,s)=&- \int_{0}^{T} \int_{\Omega} (q_1w_1-q_2w_2)(t,x) V(s,x)\, \rm{d}x\, \rm{d}t\\
	=& \int_{\Omega} v(T,x)V(s,x)\, \rm{d}x \qquad \mbox{for all}\ (T,s)\in [0,+\infty)\times[0,+\infty).
\end{align*}
Thanks to \eqref{Neumann inverse problem parabolic: intermedite 1}, we obtain
\[
\int_{0}^{T}\int_{\Omega} v(t,x) \partial_s V(s,x)\, \rm{d}x\, \rm{d}t=0.
\]
Integrating the above relation with respect to `$s$' variable, we obtain
\[
\int_{0}^{S} \int_{0}^{T}\int_{\Omega} v(t,x) \partial_s V(s,x)\, \rm{d}x \, \rm{d}t \, \rm{d}s=0.
\]
Hence
\[
\int_0^T \int_{\Omega} v(t,x)V(S,x)\, \rm{d}x\, \rm{d}t=0 \qquad \mbox{for all}\ S\geq 0.
\]
We have $v\in\mathrm{L}^2\left((0,T); H^{a}(\overline{\Omega})\right)\subset \mathrm{L}^2((0,T); \mathrm{L}^2(\Omega))$. Hence the function $\displaystyle{\int_{0}^T v(t,x)\, \rm{d}t\in \mathrm{L}^2(\Omega)}$ and 
\[
\int_{\Omega} V(S,x)\left( \int_0^T v(t,x)\, \rm{d}t\right)\, \rm{d}x=0, \quad \mbox{for all}\ S\geq 0.
\]
We fix a particular $S>0$. then If we vary $F$, the set $\big\{V(S,x): F\in C_c^{\infty}((0,T)\times\partial\Omega)\big\}$ is dense in $\mathrm{L}^2(\Omega)$ [see Theorem \ref{Quantitative Density result boundary data}]. Hence
\[
\int_0^T
v(t,x)\, \rm{d}t=0, \qquad \mbox{for all}\,  T\geq 0.
\]
We fix a particular time $t=\tilde{T}>0$. We have
\[
\int_0^{\tilde{T}+h}
v(t,x)\, \rm{d}t- \int_0^{\tilde{T}-h}
v(t,x)\, \rm{d}t=0,
\]
where $h$ is sufficiently small. It yields 
\[
\frac{1}{2h} \int_{\tilde{T}-h}^{\tilde{T}+h} v(t,x)\, \rm{d}t=0.
\]
Employing Lebesgue differentiation theorem, we obtain
\[
v(\tilde{T},x)\equiv 0, \qquad \mbox{for all}\ (\tilde{T},x)\in[0,+\infty)\times\Omega.
\]  
hence
\[
w_1\equiv w_2.
\]
From the equation \eqref{varying boundary data inverse result}, we have that 
\[
(q_1-q_2)w_1\equiv 0.
\]
We fix $t=\tilde{T}> 0$, we have that
\[
\int_{\Omega} (q_1-q_2)(x)w_1(\tilde{T},x)\, \rm{d}x=0.
\]
Again, if we vary $f\in C_c^{\infty}((0,T)\times\Omega)$, the set $\displaystyle{\big\{w_1(\tilde{T},x): f\in C_c^{\infty}((0,T)\times\partial\Omega)\big\}}$ is dense in $\mathrm{L}^2(\Omega)$ [see Theorem \ref{Density result initial data}]. Hence
\[
(q_1-q_2)(\tilde{T},x)=0, \qquad \mbox{for a.e.} \ x\in\Omega.
\]
Since $\tilde{T}>0$ is arbitrary, we obtain
\[
q_1\equiv q_2.
\]
\end{proof}

	\section{Further result:}
	In this section, we focus on a specific non-local elliptic eigenvalue Calderón problem and solve it using Pohozaev's identity, as illustrated in \cite{ros2014pohozaev}. We aim to demonstrate to the reader the significance of Pohozaev's identity in solving various non-local Calderón problems. The problem discussed in this section is of an elliptic nature.
	We consider the eigenvalue problem
	\begin{equation}\label{elliptic eigen value}
		\left \{
		\begin{aligned}
			(-\Delta)^a \phi= & \lambda \phi \qquad \qquad \mbox{in}\ (0,T)\times\Omega\\
			\phi =& 0 \qquad \qquad \mbox{in} \ (0,T)\times\mathbb{R}^N\setminus\Omega. 
		\end{aligned}
		\right .
	\end{equation}
	Here, the domain $\Omega\subset\mathbb{R}^N$ is a bounded $C^{1,1}$ domain with boundary $\partial\Omega$ and $a\in(0,1)$.
	\begin{Thm}\label{inverse result eigen value}
		Let $\phi$ is a solution to the eigenvalue problem \eqref{elliptic eigen value}. Then for any non-empty smooth open set $\Sigma\subset\partial\Omega$
		\[
		\frac{\phi}{d^a}\Big|_{\Sigma} \neq 0.
		\]
	\end{Thm}
	Before begin the proof, we begin with some available results. The set containing eigenvalues  
	\[
	\{\lambda: \ \mbox{The equation \eqref{elliptic eigen value} has a solution}\}
	\]
	is countable and discrete and the eigenvalues can be described in the following monotonic manner
	\[
	0<\lambda_1<\lambda_2\leq \lambda_3\leq \cdots\leq \lambda_n\leq \cdots +\infty.
	\]
	The following regularity of the eigenfunctions can be found in \cite{ros2014dirichlet}, \cite{fernandez2016boundary}.
	\begin{itemize}
		\item [$\bullet$] $\phi\in H^a(\mathbb{R}^N)\cap C^a(\mathbb{R}^N)$ and, for every $\beta\in[a,1+2a)$, $\phi$ is of class $C^{\beta}(\Omega)$ and 
		\begin{align}\label{Elliptic H2}
			[\phi(t,\cdot)]_{C^{\beta}(\Omega_{\delta})} \leq C \delta^{a-\beta}, \forall \delta\in(0,1).
		\end{align}
		\item [$\bullet$] The function $\displaystyle{\frac{\phi}{d^a}}\Big|_{\Omega}$ can be continuously extended to $\overline{\Omega}$. Moreover, there exists $\alpha\in(0,1)$ such that $\displaystyle{\frac{\phi}{d^a}\in C^{\alpha}(\overline{\Omega})}$. In addition, for $\beta\in[\alpha,a+\alpha]$, the following estimate holds
		\begin{align}\label{H3 Elliptic}
			\left[\frac{\phi}{d^a}\right]_{C^{\beta}(\Omega_{\delta})} \leq C \delta^{a-\beta}, \forall \delta\in(0,1).
		\end{align}
	\end{itemize}
	Hence Pohozaev identity \eqref{Phozaev identity} holds for $\phi$ \cite{ros2014pohozaev}.
	\begin{align}\label{Phozaev identity phi}
		\int_{\Omega}(x\cdot \nabla \phi)(-\Delta)^a \phi \, \rm{d}x= \frac{2a-N}{2}\int_{\Omega}&\phi(-\Delta)^a \phi \, \rm{d}x\\
		&- \frac{\Gamma(1+a)^2}{2}\int_{\partial\Omega}\left(\frac{\phi}{d^a}\right)^2(x\cdot\nu) \, \rm{d}\sigma,\nonumber
	\end{align}
	where $\nu$ is the unit outward normal to $\partial\Omega$ at the point $x$ and $\Gamma$ is the gamma function. We state another version of Pohozaev identity. Let $\theta$ be another function such that $(-\Delta)^a\theta$ defined pointwise and $\displaystyle{\frac{\theta}{d^a}}$ continuously extended to boundary. Furthermore, let $\theta$ satisfy \eqref{Elliptic H2} and \eqref{H3 Elliptic}, then
	\begin{align}\label{weak pohozaev}
		\int_{\Omega} \theta_{x_i}(-\Delta)^a\phi  \, \rm{d}x= -\int_{\Omega} \phi_{x_i}&(-\Delta)^a\theta \, \rm{d}x\\
		-& \Gamma(1+a)^2\int_{\partial\Omega} \frac{\phi}{d^a}\frac{\theta}{d^a}\nu_i\, \rm{d}\sigma.\nonumber
	\end{align}
	Here the subscript $x_i$ denotes the derivative of the function with respect to the $i$th coordinate and $\nu_i$ is the $i$th component of the unit normal vector $\nu$ at the boundary. Proof of this can be found in \cite{ros2014pohozaev}.
	\newline
	Let $\alpha=(\alpha_1,\cdot,\alpha_N)\in\mathbb{R}^N$. For $1\leq i\leq N$, the function $\theta=\alpha_i \phi$ is an admissible function in \eqref{weak pohozaev}. It yields
	\begin{align*}
		\alpha_i\int_{\Omega} \phi_{x_i}(-\Delta)^a\phi  \, \rm{d}x= -\alpha_i\int_{\Omega} \phi_{x_i}&(-\Delta)^a\phi \, \rm{d}x\\
		-& \Gamma(1+a)^2\int_{\partial\Omega} \left(\frac{\phi}{d^a}\right)^2(\alpha_i\nu_i)\, \rm{d}\sigma.
	\end{align*}
	It implies
	\[
	\alpha_i\int_{\Omega} \phi_{x_i}(-\Delta)^a\phi  \, \rm{d}x=-\frac{\Gamma(1+a)^2}{2}\int_{\partial\Omega} \left(\frac{\phi}{d^a}\right)^2(\alpha_i\nu_i)\, \rm{d}\sigma, \quad \forall \, i=1,\cdots,N.
	\]
	Summing over the index `$i$' yields
	\[
	\int_{\Omega} (\alpha\cdot\nabla\phi)(-\Delta)^a\phi= -\frac{\Gamma(1+a)^2}{2}\int_{\partial\Omega} \left(\frac{\phi}{d^a}\right)^2(\alpha\cdot\nu)\, \rm{d}\sigma.
	\] 
	Subtracting from \eqref{Phozaev identity phi}, we obtain
	\begin{align}\label{Phozaev identity with alpha}
		\int_{\Omega}((x-\alpha)\cdot \nabla \phi)&(-\Delta)^a \phi \, \rm{d}x\\=& \frac{2a-N}{2}\int_{\Omega}\phi(-\Delta)^a \phi \, \rm{d}x
		- \frac{\Gamma(1+a)^2}{2}\int_{\partial\Omega}\left(\frac{\phi}{d^a}\right)^2((x-\alpha)\cdot\nu) \, \rm{d}\sigma, \quad \forall \, \alpha\in\mathbb{R}^N. \nonumber
	\end{align} 
	Next we move onto the proof of theorem \ref{inverse result eigen value}.
	
	\begin{proof}
		Let $\Omega\subset\mathbb{R}^N$ be a bounded $C^{1,1}$ domain with $\partial\Omega$ as it's boundary. Let $\Sigma\subset\partial\Omega$ be a non-empty open set. Then there exists $\alpha\in\mathbb{R}^N$, such that \cite{uhlmann2014inverse} 
		\begin{align*}
			\Sigma:=& \left\{ x\in\partial\Omega, (x-\alpha)\cdot\nu\geq 0\right\}\\
			\partial\Omega\setminus\Sigma:=& \left\{ x\in\partial\Omega, (x-\alpha)\cdot\nu\leq 0\right\}.
		\end{align*}
		We use the relation \eqref{Phozaev identity with alpha}:
		\begin{align*}
			\int_{\Omega}((x-\alpha)\cdot \nabla \phi)&(-\Delta)^a \phi \, \rm{d}x\\=& \frac{2a-N}{2}\int_{\Omega}\phi(-\Delta)^a \phi \, \rm{d}x
			- \frac{\Gamma(1+a)^2}{2}\int_{\partial\Omega}\left(\frac{\phi}{d^a}\right)^2((x-\alpha)\cdot\nu) \, \rm{d}\sigma. 
		\end{align*} 
		Plugging $(-\Delta)^a\phi=\lambda\phi$ in the above relation, we obtain
		\begin{align*}
			\int_{\Omega} &\lambda(x-\alpha)\cdot \nabla\left(\frac{\phi^2}{2}\right)\, \rm{d}x= \frac{2a-N}{2} \int_{\Omega} \lambda \phi^2\, \rm{d}x\\
			&- \frac{\Gamma(1+a)^2}{2}\int_{\Sigma}\left(\frac{\phi}{d^a}\right)^2((x-\alpha)\cdot\nu) \, \rm{d}\sigma- \frac{\Gamma(1+a)^2}{2}\int_{\partial\Omega\setminus\Sigma}\left(\frac{\phi}{d^a}\right)^2((x-\alpha)\cdot\nu) \, \rm{d}\sigma.
		\end{align*}
		We prove via contradiction. Let $\displaystyle{\frac{\phi}{d^a}\Big|_{\Sigma}= 0}$. Employing Gauss divergence theorem on the first term and thanks to $\displaystyle{\phi\equiv0}$ outside $\Omega$, we obtain the following relation.
		\begin{align*}
			a\lambda \int_{\Omega} \phi^2 \, \rm{d}x=   \frac{\Gamma(1+a)^2}{2}\int_{\partial\Omega\setminus\Sigma}\left(\frac{\phi}{d^a}\right)^2((x-\alpha)\cdot\nu) \, \rm{d}\sigma \leq 0
		\end{align*}
		As $a,\lambda>0$, it yields $\phi\equiv 0$, which contradicts $\lambda$ is an eigenvalue. Hence for any open subset $\Sigma\subset\partial\Omega$
		\[
		\frac{\phi}{d^a}\Big|_{\Sigma} \neq 0.
		\]
	\end{proof}

	\vspace{.3cm}

	\textbf{Acknowledgment:} 
	The author was funded by the Department of Atomic Energy $($DAE$)$, Government of India, in the form of Postdoctoral Research Fellowship.
	
	The author sincerely expresses gratitude to Prof. Tuhin Ghosh (Harish-Chandra Research Institute, Prayagraj, Uttar Pradesh, India) for his valuable insights and feedback, which significantly contributed to improving this article.
	

	\vspace{.3cm}
	\textbf{Data Availability:} Data sharing is not applicable to this article as no datasets were generated or analyzed during the current study. 
	
	\bibliography{inverse.bib}
	
\section{Appendix} \label{observability wave}
	
This proof is motivated from the article \cite{biccari2025boundary}. We follow exactly the same method as in \cite{biccari2025boundary}. We only have an extra non-negative small potential with small growth. The smallness in our calculations is explicit in nature. Consider the wave equation: 
\begin{equation}\label{adjointWave}
	\begin{cases}
		p_{tt} + (-\Delta)^ap+qp = 0 & \mbox{in }\; (0,T)\times\Omega
		\\
		p=0 & \mbox{in }\; (0,T)\times(\mathbb R^N\setminus\Omega)
		\\
		p(\cdot,0)=p_0,\;\;p_t(\cdot,0)=p_1 & \mbox{in }\;\Omega.
	\end{cases}
\end{equation}
Note that, at this level, the system being time-reversible, taking the initial data at $t = 0$ or $t = T$ is irrelevant, contrarily to the parabolic setting.

The system \eqref{adjointWave}, governed by the non-local fractional operator $(-\Delta)^a$, which is symmetric maximal monotone operator and generates group of isometries. As a consequence, it admits a unique weak solution that preserves the total energy (see, for instance, [Chapter 10, Theorem 10.14] in \cite{brezis2011functional}). In other words, for every initial data $(p_0,p_1)\in H_0^a(\Omega)\times L^2(\Omega)$ there exists a unique finite energy solution 
\begin{align*}
	p\in C([0,T];H_0^a(\Omega))\cap C^1([0,T];\mathrm{L}^2(\Omega))\cap C^2([0,T];H^{-a}(\Omega))
\end{align*}
of \eqref{adjointWave}. Here $H_0^a(\Omega)$ denotes the fractional order Sobolev space consisting of all functions in $H^a(\mathbb{R}^N)$ vanishing in $\mathbb{R}^N\setminus\Omega$, while $H^{-a}(\Omega)$ is its dual. Moreover, the function $p$ admits the following spectral representation:
\begin{align*}
	p(x,t) = \sum_{j\in\mathbb{N}} p_j(t)\phi_j(x)
\end{align*}
where, for each \( j \in \mathbb{N} \), the coefficient function $p_j$ satisfies the second-order ordinary differential equation:
\begin{align*}
	\begin{cases}
		p_j^{\prime\prime}(t)+ \lambda_j p_j(t) = 0,\qquad t\in (0,T),
		\\
		p_j(0) = a_j, \;\;\; p_j'(0) = b_j.
	\end{cases}
\end{align*}
Here, $(\lambda_j,\phi_j)_{j\in\mathbb{N}}$ denote the eigenvalues and eigenfunctions of the operator $(-\Delta)^a+q$, and $(a_j,b_j)_{j\in\mathbb{N}}$ are the Fourier coefficients of the initial data $(p_0,p_1)$ with respect to the basis $(\phi_j)_{j\in\mathbb{N}}$, that is,
\begin{align*}
	p_0(x) = \sum_{j\in\mathbb{N}} a_j\phi_j(x) \quad \mbox{and} \quad p_1(x) = \sum_{j\in\mathbb{N}} b_j\phi_j(x).
\end{align*}
Thanks to the symmetricity of the operator $(-\Delta)^a+q$, the energy of solutions of \eqref{adjointWave} is conserved along time.  That is, 
\begin{align}\label{e24}
	E_a(t):=\frac 12\int_{\Omega}|p_t|^2\,dx+\frac 12\int_{\mathbb{R}^N}|(-\Delta)^{\frac a2}p|^2\,dx+ \frac 12\int_{\Omega} qp^2\, dx= E_a(0).
\end{align}
Now we establish the  multiplier identity for solutions of \eqref{adjointWave}. To this end, we focus on a special class of solutions $p\in \mathcal H_J$, with $\mathcal H_J$  as in \eqref{finite projective space}. They are finite energy solutions of \eqref{adjointWave} involving a finite number of Fourier components, of the form 
\begin{align}\label{solFourier}
	p(x,t) = \sum_{j=1}^J p_j(t)\phi_j(x).
\end{align}
The following result is the Pohozaev identity for the finite energy solution of the wave equation \eqref{adjointWave} in the projective space \eqref{finite projective space}, which helps us obtain Proposition \ref{observability estimate}. 
\begin{Prop}\label{lemmaMult}
	Let $a\in(0,1)$ and $p\in\mathcal H_J$ be the solution of \eqref{adjointWave} constituted by a finite number of Fourier components, as defined in \eqref{solFourier}. Then,
	\begin{align}\label{e210}
		\frac{\Gamma(1+a)^2}{2}\int_{\epsilon}^T \int_{\partial\Omega}(x\cdot\nu)\left|\frac{p}{d^a}\right|^2\,d\sigma dt
		&=  aTE_a(\epsilon)+ \int_{\Omega}p_t\left(x\cdot\nabla p + \frac{N-a}{2}p\right)\,dx\,\bigg|_{t=\epsilon}^{t=T}\\
		&-\int_{\epsilon}^T\!\int_{\Omega}\left(q\frac a2+x\cdot\nabla \frac{q}{2}\right) p^2 \, dxdt- aT\int_{\Omega} qp^2 \, dx\Big|_{t=\epsilon} \notag
	\end{align}
	where $\Gamma$ is the Euler-Gamma function.
\end{Prop}
\begin{proof}
	\noindent\textbf{Step 1: Application of Pohozaev identity.} We formally multiplying the first equation in \eqref{adjointWave} by $x\cdot\nabla p$. Integrating over $(\epsilon,T)\times\Omega$ we get that
	\begin{align}\label{A1}
		\int_{\epsilon}^T\int_{\Omega}p_{tt}\left(x\cdot\nabla p\right)\,dxdt +\int_{\epsilon}^T\int_{\Omega}(-\Delta)^ap\left(x\cdot\nabla p\right)\,dxdt+\int_{\epsilon}^T\int_{\Omega} q(x\cdot \nabla p)p =0.
	\end{align}
	The regularity assumptions of \cite{ros2014pohozaev}, are fulfilled by solutions of the fractional wave equation involving a finite number of Fourier components \cite{biccari2025boundary}, we have that, for all $t\in [\epsilon, T]$, 
	\begin{align}\label{A2}
		\int_{\Omega}(-\Delta)^ap \left(x\cdot\nabla p\right)\,dx = \frac{2a-N}{2}\int_{\mathbb{R}^N}|(-\Delta)^{\frac a2}p|^2\,dx - \frac{\Gamma(1+a)^2}{2}\int_{\partial\Omega}(x\cdot\nu)\left|\frac{p}{d^a}\right|^2\,d\sigma.
	\end{align}
	Using \eqref{A2} we get from \eqref{A1} that 
	\begin{align}\label{A3}
		\int_{\epsilon}^T\int_{\Omega}p_{tt}\left(x\cdot\nabla p\right)\,dxdt + \int_{\epsilon}^T\int_{\Omega} qp(x\cdot \nabla p) \, dxdt =& \frac{N-2a}{2}\int_{\epsilon}^T\int_{\mathbb{R}^N}|(-\Delta)^{\frac a2}p|^2\, dxdt\\
		+ &\frac{\Gamma(1+a)^2}{2}\int_{\epsilon}^T\int_{\partial\Omega}(x\cdot\nu)\left|\frac{p}{d^a}\right|^2\,d\sigma dt. \nonumber
	\end{align}
	The right hand side can be computed as follows:
	\begin{align}\label{A4}
		\int_{\epsilon}^T\!\int_{\Omega}p_{tt}&\left(x\cdot\nabla p\right)\,dxdt + \int_{\epsilon}^T\int_{\Omega} q(x\cdot \nabla p)p \, dx\\ = &\, \int_{\Omega}p_t\left(x\cdot\nabla p\right)\,dx\,\bigg|_{t=\epsilon}^{t=T} 
		- \int_{\epsilon}^T\!\int_{\Omega} p_t\left(x\cdot\nabla p_t\right)\,dxdt+\int_{\epsilon}^T\int_{\Omega} q(x\cdot \nabla p)p \, dx  \notag
		\\
		= & \int_{\Omega}p_t\left(x\cdot\nabla p\right)\,dx\,\bigg|_{t=\epsilon}^{t=T} 
		-\int_{\epsilon}^T\!\int_{\Omega}x\cdot\nabla\left(\frac{|p_{t}|^2}{2}\right)\,dxdt+\int_{\epsilon}^T\int_{\Omega} qx\cdot \nabla \left(\frac{p^2}{2}\right) \, dx  \notag
		\\
		= & \int_{\Omega} p_t\left(x\cdot\nabla p\right)\,dx\,\bigg|_{t=\epsilon}^{t=T} + \frac N2\int_{\epsilon}^T\!\int_{\Omega}|p_t|^2\,dxdt- \int_{\epsilon}^T\!\int_{\Omega}\left(q\frac N2+x\cdot\nabla \frac{q}{2}\right) p^2 \, dxdt \label{e22}  .
	\end{align}
	Once identities \eqref{A3} and \eqref{A4} are obtained and justified, combining them, we get \eqref{e22}. We refer to \cite{biccari2025boundary} for justification. The same method follows when we have a smooth compactly supported potential. 
	
	\noindent\textbf{Step 2: Equipartition of the energy}. We now derive the equipartition of the energy identity. Multiplying the first equation in \eqref{adjointWave} by $p$ and integrating by parts over $(\epsilon, T)\times\Omega$ (by using the integration by parts formula \eqref{integration by parts Hs}) we obtain 
	\begin{align}\label{e28}
		-\int_{\epsilon}^T\int_{\Omega}|p_t|^2\,dxdt + \int_{\epsilon}^T\int_{\mathbb{R}^N}|(-\Delta)^{\frac a2}p|^2\,dxdt + \int_{\epsilon} p_tp\,dx\,\bigg|_{t=\epsilon}^{t=T}+\int_{\epsilon}^T\int_{\Omega}q|p|^2\,dxdt=0.
	\end{align}
	Moreover, the conservation of energy yields
	\begin{align*}
		\frac N2 \int_{\epsilon}^T \int_{\Omega}|p_t|^2\,dxdt + & \frac{2a-N}{2}\int_{\epsilon}^T\int_{\mathbb{R}^N} |(-\Delta)^{\frac a2}p|^2\,dxdt
		\\
		= \frac a2\int_{\epsilon}^T\int_{\Omega}|p_t|^2\,dxdt &+ \frac a2\int_{\epsilon}^T\int_{\mathbb{R}^N} |(-\Delta)^{\frac a2}p|^2\,dxdt + \frac{N-a}{2}\int_{\epsilon}^T\int_{\Omega}\Big(|p_t|^2 - |(-\Delta)^{\frac a2}p|^2\Big)\,dxdt
		\\
		= aTE_a(\epsilon) +&  \frac{N-a}{2}\int_{\epsilon}^T\int_{\Omega}|p_t|^2 - \frac{N-a}{2}\int_{\epsilon}^T\int_{\mathbb{R}^N} |(-\Delta)^{\frac a2}p|^2\,dxdt -aT\int_{\Omega} qp^2 \, dx\Big|_{t=\epsilon}.
	\end{align*}
	Using this in the identity \eqref{A3}, we then have
	\begin{align}\label{e29}
		\frac{\Gamma(1+a)^2}{2}&\int_0^T\int_{\partial\Omega}(x\cdot\nu)\left|\frac{p}{d^a}\right|^2\,d\sigma dt =\, aTE_a(\epsilon)\\
		+& \frac{N-a}{2}\int_0^T\int_{\Omega}|p_t|^2- \frac{N-a}{2}\int_{\epsilon}^T\int_{\mathbb{R}^N}|(-\Delta)^{\frac a2}p|^2\,dxdt 
		 + \int_{\Omega} p_t\left(x\cdot\nabla p\right)\,dx\,\bigg|_{t=0}^{t=T}\notag \\
		 -& \int_{\epsilon}^T\!\int_{\Omega}\left(q\frac N2+x\cdot\nabla \frac{q}{2}\right) p^2 \, dxdt- aT\int_{\Omega} qp^2 \, dx\Big|_{t=\epsilon}. \notag
	\end{align}
	Combining \eqref{e28} and \eqref{e29} we get the identity \eqref{e210} and the proof is finished.
\end{proof}
Next we move to the proof of the observability estimate for the finite dimensional solution of the wave equation \eqref{adjointWave}. 
\begin{Lem}\label{prop31}
	Fix $a\in(\frac 12,1)$ and $J\in\mathbb{N}$. Let $\displaystyle{\theta< \frac{1}{2}\left(1+C_{HS}\left(\frac N2+R\right)\right)^{-1}}$, where $C_{HS}$ is the Hardy-Sobolev constant on the domain $\Omega$ and $R:=\max\{\Vert x\Vert, x\in\Omega\}$. Let the potential $q$ is non-negative and  $|q(x)|,|\nabla q(x)|\leq \theta$ for all $x\in\Omega$.  Set
	\begin{equation}\label{e39}
		T_0(J) := C\lambda_J^{1-a}, 
	\end{equation}
	with a constant $C=C(N,a,\Omega, \Vert q\Vert_{\mathrm{L}^{\infty}(\Omega)})>0$.
	Then, for all $T>T_0(J)$, the following boundary observability inequality holds:
	\begin{align}\label{e37}
		E_a(T^*)\le \frac{\Gamma(1+a)^2}{2a(T-T_0(J))}\int_{\epsilon}^T\int_{\partial\Omega}(x\cdot\nu)\left|\frac{p}{d^a}\right|^2\,d\sigma dt,
	\end{align}
	for all solution $p\in \mathcal H_J$ of \eqref{adjointWave} defined in \eqref{solFourier} and for all $T^*\in[0,T]$. 
\end{Lem}
\begin{proof}
	$a\in(1/2,1)$. From \eqref{e210} we get that
	\begin{align*}
		\frac{\Gamma(1+a)^2}{2}\int_{\epsilon}^T \int_{\partial\Omega}(x\cdot\nu)\left|\frac{p}{d^a}\right|^2\,d\sigma dt
		&=  aTE_a(\epsilon)+ \int_{\Omega}p_t\left(x\cdot\nabla p + \frac{N-a}{2}p\right)\,dx\,\bigg|_{t=\epsilon}^{t=T}\\
		&-\int_{\epsilon}^T\!\int_{\Omega}\left(q\frac a2+x\cdot\nabla \frac{q}{2}\right) p^2 \, dxdt- aT\int_{\Omega} qp^2 \, dx\Big|_{t=\epsilon} \notag
	\end{align*}
	All the terms above is well defined \cite{biccari2025boundary}.  Let's define:
	\begin{align*}
		\xi(t):= \int_{\Omega}p_t\left(x\cdot\nabla p + \frac{N-a}{2}p\right)\,dx,
	\end{align*} 
	This is a well defined function. We borrow the following result from \cite{biccari2025boundary}. 
	\[
	\left|\xi(t)\right|\leq \texttt{C} \lambda_J^{1-a} E_a(\epsilon),
	\]
	where $\texttt{C}$ is a positive constant depending on $N,s$ and the domain $\Omega$. Thanks to Hardy Sobolev estimate, the following estimate holds
	\[
	\int_{\Omega} p^2(T) \, dx \leq  C_{HS} \int_{\Omega}(-\Delta)^{\frac a2} p^2(T) \, dx \leq E_a(\epsilon).
	\]
	Hence,
	\begin{align*}
		\left|\int_{\epsilon}^T\!\int_{\Omega}\left(q\frac a2+x\cdot\nabla \frac{q}{2}\right) p^2 \, dxdt\right| \leq & T\theta C_{HS} \left(\frac N2+R\right) E_a(\epsilon)\\
		\left|aT\int_{\Omega} qp^2 \, dx\Big|_{t=\epsilon}\right| \leq aT \theta E_s(\epsilon).
	\end{align*}
	All the above estimate yields:
	\[
	a\left( T-T\theta-T\theta C_{HS}\left(\frac {N}{2}+R\right) -2\texttt{C}\lambda_J^{1-a}\right) E_a(\epsilon) \leq \frac{\Gamma(1+a)^2}{2}\int_{\epsilon}^T\int_{\partial\Omega}(x\cdot\nu)\left|\frac{p}{d^a}\right|^2\,d\sigma dt
	\]
	From assumption we have $\displaystyle{\theta< \frac{1}{2}\left(1+C_{HS}\left(\frac N2+R\right)\right)^{-1}}$. Choosing the positive constant $C:=4\texttt{C}$, we conclude our proof.
\end{proof}
Thanks to Lemma \ref{prop31} and Proposition \ref{lemmaMult}, we can recover the results outlined in \cite{biccari2025boundary} and thereby establish Proposition \ref{observability estimate}.

\end{document}